  \documentclass[letterpaper,oneside]{article}

  \usepackage{amssymb,amsmath,amsthm,epsfig,mathrsfs}
  \usepackage[dvips,letterpaper,margin=1.3in]{geometry} 
  \usepackage{xspace, xcolor} 
  \usepackage{tikz}
  \usetikzlibrary{arrows,calc,decorations.pathmorphing,backgrounds,positioning,fit}
  \usepackage[
     breaklinks,colorlinks,
     citecolor=teal,linkcolor=magenta,urlcolor=teal
  ]{hyperref}
  \newcommand{\s}{\ensuremath{\mathcal{S}}\xspace}            % surface S
  \newcommand{\mcg}{\ensuremath{\mathcal{MCG}(\s)}\xspace}    % MCG(S)
  \newcommand{\cc}{\ensuremath{\mathcal{C}(\s)}\xspace}       % C(S)
  \newcommand{\mg}{\ensuremath{\Mark(\s)}\xspace}              % MG(S)
  \newcommand{\m}{\ensuremath{\mathcal{M}}}                   % rep markings
  \newcommand{\nn}{\ensuremath{\mathbb{N}}}                    % N
  \newcommand{\zz}{\ensuremath{\mathbb{Z}}}                    % Z
  \DeclareMathOperator{\ord}{order}                            % order
  \DeclareMathOperator{\Mark}{Mark}                            % mark
  \DeclareMathOperator{\base}{base}                            % base
  \DeclareMathOperator{\tran}{trans}                           % trans
  \DeclareMathOperator{\diam}{diam}                            % diameter
  \DeclareMathOperator{\dist}{Dist}                            % distance
  \DeclareMathOperator{\fix}{Fix}                              % Fix
  \newcommand{\wfix}{\widetilde{\fix}}                         % sym
  \DeclareMathOperator{\stab}{Stab}                            % stabilizer
  \DeclareMathOperator{\homeo}{Homeo}                          % Homeo
  \DeclareMathOperator{\cat}{CAT}                              % CAT
  \DeclareMathOperator{\Sl}{SL}                                % SL
  \newtheoremstyle{introtheorem}{.5em}{.5em}{\itshape}{}{\bfseries\itshape}{.}{.5em}{}

  \theoremstyle{introtheorem}
    \newtheorem{introthm}{Theorem}
    
    \newtheorem{introcor}[introthm]{Corollary}

    \theoremstyle{plain}
    \newtheorem{theorem}{Theorem}[subsection]
    \newtheorem*{theorem*}{Theorem}
    \newtheorem{proposition}[theorem]{Proposition}
    \newtheorem{corollary}[theorem]{Corollary}
    \newtheorem{lemma}[theorem]{Lemma}

  \theoremstyle{definition}
    \newtheorem{definition}[theorem]{Definition}

    \newtheorem{remark}[theorem]{Remark}
    \newtheorem{notation}[theorem]{Notations}

  \newcommand{\thmref}[1]{Theorem~\ref{#1}}
  \newcommand{\propref}[1]{Proposition~\ref{#1}}
  \newcommand{\lemref}[1]{Lemma~\ref{#1}}
  
  \newcommand{\corref}[1]{Corollary~\ref{#1}}
  \newcommand{\figref}[1]{Figure~\ref{#1}}
  
  \newcommand{\remref}[1]{Remark~\ref{#1}}
  \newcommand{\defref}[1]{Definition~\ref{#1}}
  
  \newcommand{\secref}[1]{\S\ref{#1}}

  \newcommand\address[1]{}
  \newcommand\email[1]{}
  \newcommand\dedicatory[1]{}

\begin{document}

%section{Title and Abstract}

  \title      {  Linearly Bounded Conjugator Property for Mapping Class Groups  }

  \author     {  Jing Tao  }

  \address    {  Department of Mathematics,
                 University of Illinois at Chicago,
                 851 South Morgan Street,
                 Chicago, IL 60607-7045, USA  }

  \email      {  jingtao@math.uic.edu  }

  \date  {} % {  Preliminary version --- \today  }

  \maketitle
  
  \begin{abstract}

    Given two conjugate mapping classes $f$ and $g$, we produce a conjugating
    element $\omega$ such that $| \omega | \le K \big( | f | + | g |
    \big)$, where $| \cdot |$ denotes the word metric with respect to a
    fixed generating set, and $K$ is a constant depending only on the
    generating set. As a consequence, the conjugacy problem for mapping
    class groups is exponentially bounded.

  \end{abstract}

  %\tableofcontents

  %\newpage

\section{Introduction}

  Two fundamental problems in group theory posed by Dehn are the word
  problem and the conjugacy problem \cite{Deh11}.  Given a group with a
  fixed presentation, the \emph{word problem} asks if there is an algorithm
  that can decide in finite time if a given word is the identity. The
  \emph{conjugacy problem} seeks an algorithm to decide if two words
  represent the same conjugacy class. Since the conjugacy class of the
  identity element is itself, the word problem can be seen as a special
  case of the conjugacy problem. Not all groups have solvable word problem
  \cite{Nov58,Boo59}, hence the same is true for the conjugacy problem. 
  
  In this paper, we are interested in these problems for mapping class
  groups $\mcg$ of surfaces $\s$ of finite type. We establish the
  following:

  \begin{introthm}\label{introthmComplexity}

    There is an exponential-time algorithm to solve the conjugacy problem
    for \mcg.

  \end{introthm}
  
  There is some history to the word and conjugacy problems for \mcg. The
  first solution to the word problem can be attributed to Grossman, whose
  actual contribution is proving residual finiteness for \mcg \cite{Gro74}.
  In \cite{Mos95}, Mosher showed \mcg admits an automatic structure, from
  which a quadratic-time solution to the word problem is obtained. (See
  \cite{ECH92} for a background on automatic groups. It is not yet known if
  a sub-quadratic solution is possible.) In \cite{Hem79}, Hemion solved the
  conjugacy problem for \mcg, but his algorithm is not exponentially
  bounded.  In \cite{Mos86}, Mosher gave a faster algorithm for deciding
  conjugacy among pseudo-Anosov mapping classes. (A similar result was
  recently obtained by Agol in \cite{Ago11}.) Using the work of
  Bestvina-Handel \cite{BH95}, which gives an algorithm for detecting
  pseudo-Anosov mapping classes, Mosher extended his result to compute
  complete conjugacy invariants for all mapping classes \cite{Mos03}. 
  
  Our strategy to prove \thmref{introthmComplexity} is to apply Mosher's
  automaticity result. In general, a solution to the word problem does not
  necessarily yield a solution to the conjugacy problem: it is an open
  question whether all automatic groups have solvable conjugacy problem
  \cite{ECH92}. A sufficient condition is if the group has \emph{linearly
  bounded conjugator (L.B.C.) property} (see theorem below or
  \defref{defLBC}). The main theorem of our paper is that L.B.C.~property
  is satisfied by \mcg. This answers a question in \cite{Far06}.

  \begin{introthm}[L.B.C.~property for \mcg]\label{introthmLBC}

   Let $\Lambda$ be a finite generating set for \mcg. There exists a
   constant $K$, depending only on $\Lambda$, such that if $f, g \in \mcg$
   are conjugate, then there is a conjugating element $\omega$ with \[
   |\omega| \le K \big( |f| + |g| \big). \]

  \end{introthm}
  
  To see how \thmref{introthmComplexity} follows from \thmref{introthmLBC},
  we give an algorithm to the conjugacy problem. Given two arbitrary
  elements $f, g \in \mcg$, let $B$ be the ball of radius $K \big( |f| +
  |g| \big)$ in \mcg. To decide if $f$ and $g$ are conjugate it suffices to
  check if $\omega \in B$ satisfies $\omega f\omega^{-1}g^{-1}=1$. We run
  Mosher's quadratic algorithm to the word problem to all words of the form
  $\omega f \omega^{-1}g^{-1}$ with $\omega \in B$. The number of elements
  in $B$ is an exponential function of the radius, therefore the complexity
  of this solution is an exponential function of the word lengths of $f$
  and $g$.
  
  Linearly-bounded conjugator property is satisfied by hyperbolic groups
  \cite[Lemma 10]{Lys89}, as well as by torsion elements in groups acting on
  CAT(0) spaces \cite[III.1.13]{BH99}. These are important classes of groups
  which have solvable word and conjugacy problems \cite{Gro87,BH99}. Hyperbolic
  groups in fact have efficient algorithms: the word problem is solvable in
  linear time, and the conjugacy problem in quadratic time \cite{BH99}. As
  long as the surface \s has disjoint isotopy classes of curves, \mcg is
  not hyperbolic, as Dehn twists about disjoint curves give rise to higher
  rank free abelian subgroups. It is also known that \mcg does not act on
  any complete $\cat(0)$ space \cite[II.7.26]{BH99}. Nevertheless, \mcg
  shares many properties with hyperbolic groups, and much of the pursuit in
  its study has been to understand to what extent it resembles and differs
  from hyperbolic groups. Establishing L.B.C.~property for \mcg thus
  provides another positive analogy between $\mcg$ and hyperbolic groups.

  After we announced our result, Hamenst\"adt \cite{Ham09} announced
  biautomaticity for \mcg, which generalizes Mosher's automaticity result
  as well as obtains \thmref{introthmComplexity}. Another consequence of
  her work is the exponentially-bounded conjugator property for \mcg.
  Notice, however, that this bound only gives a doubly-exponential solution
  to the conjugacy problem if we use the same algorithm as described below
  \thmref{introthmLBC}, since the search space for the conjugator would
  grow doubly-exponential in terms of the word lengths of the elements.

  \subsection{Idea of the proof of \thmref{introthmLBC}}
  
  The proof of \thmref{introthmLBC} is broken up into three arguments,
  following the classification of the elements of \mcg into pseudo-Anosov,
  reducible, and finite order. The case of the pseudo-Anosov elements was
  settled by Masur-Minsky in \cite{MM00}, using the machinery of hierarchies
  developed in the same paper. This paper resolves the other two cases.

  Surprisingly, it turns out the most delicate case involves the finite
  order elements of \mcg. In many ways, pseudo-Anosov elements of \mcg can
  be viewed as the ``hyperbolic'' elements of \mcg, whereas the finite
  order elements are the ``elliptic'' ones. The methods that Masur and
  Minsky developed are suited for elements which behave more hyperbolically,
  and thus are not effective for the finite order elements. Our main
  contribution is the development of new tools for the study of finite
  order mapping classes. Just as in the case of pseudo-Anosov mapping
  classes, we rely heavily on the machinery of hierarchies, which we need
  to extend so it is more suited to deal with the elliptic geometry.

  We briefly explain how hierarchies are related to words in \mcg. A
  natural model space for \mcg is the marking graph \mg of \s. A marking
  $\mu_B \in \mg$ is a collection of curves on $\s$ satisfying certain
  technical conditions (see \secref{secBackground} for a precise
  definition). Given an element $f \in \mcg$, the image of $\mu_B$ under
  $f$ determines $f$ up to finitely many choices. Being a model space,
  paths from $\mu_B$ to $f\mu_B$ in \mg are naturally associated to words
  representing $f$, and the distance between $\mu_B$ and $f\mu_B$ is
  comparable to the word length of $f$ (fixing a generating set for \mcg).
  Thus to understand the word length of $f$ is the same as understanding
  efficient paths from $\mu_B$ to $f\mu_B$.
  
  Even though $\mg$ (or any other model space) is not hyperbolic or
  $\cat(0)$, there is a coarsely well-defined projection map from $\mg$ to
  a product of hyperbolic spaces $\prod_Z \mathcal{C}(Z)$: each factor
  $\mathcal{C}(Z)$ is the curve complex of a subsurface $Z$ of $\s$, and
  the product is taken over all essential (possibly annular) subsurfaces of
  \s. The fact that each $\mathcal{C}(Z)$ is hyperbolic was established by
  \cite{MM99}. For each such $Z$, the projection map $\pi_Z \colon \mg \to
  \mathcal{C}(Z)$ is obtained by a surgery procedure (see
  \secref{secBackground}). By connecting $\pi_Z(\mu_B)$ and $\pi_Z(f\mu_B)$
  by a geodesic path in $\mathcal{C}(Z)$, one can associate to each element
  $f \in \mcg$ a family of geodesics in curve complexes. These geodesics
  are organized by hierarchies to produce efficient paths connecting
  $\mu_B$ to $f\mu_B$ in \mg. An important consequence of hierarchies is
  the Distance Formula (\thmref{thmDF}), which states that the distance
  between $\mu_B$ and $f\mu_B$ is well approximated by the sum of the curve
  complex distances between their projections, where the sum is taken over
  those subsurfaces to which the projections are sufficiently far apart.

  The hyperbolic geometry of the pseudo-Anosov elements of \mcg is
  exhibited in the fact that they act hyperbolically (with north-south
  dynamics) on the curve complex \cc of \s \cite{MM99,MM00}. This is
  analogous to the way how the infinite order elements of a hyperbolic
  group act on a Cayley graph of the group. Hierarchies were used to build
  quasi-axes in \cc for the action of pseudo-Anosov mapping classes. There
  is also a fellow-traveling type property for hierarchies which applies to
  fellow-traveling quasi-axes. These facts allowed Masur and Minsky to
  extend the proof of L.B.C.~property from infinite-order elements of
  hyperbolic groups to pseudo-Anosov elements of \mcg \cite[Theorem
  7.2]{MM00}.  

  The appropriate analogy for the finite order elements of \mcg are the
  torsion elements of a group $G$ which acts properly and cocompactly on a
  $\cat(0)$ (or hyperbolic) space $X$. We are inspired by the argument
  contained in \cite{BH99} on L.B.C.~property for torsion elements of $G$
  which we will briefly sketch. Let $x \in X$ be a fixed base point. We say
  an element $g \in G$ acts \emph{elliptically} on $X$ if it satisfies two
  conditions. First, $g$ acts on $X$ with (coarse) fixed points. Second,
  the distance from $x$ to the center of mass of the orbit of $x$ under
  $\langle g \rangle$ is comparable to the word length of $g$. When $X$ is
  $\cat(0)$ (or hyperbolic), $g$ is torsion implies $g$ acts elliptically.
  After conjugating $g$ by an appropriate element, its center of mass can
  be moved into a fixed ball containing a fundamental domain for the action
  of $G$ on $X$. The set of torsion elements of $G$ having fixed points
  inside the ball is finite and contains a representative for each
  conjugacy class. From here, one can reduce L.B.C.~property in the
  elliptic case to a finite set. 

  To establish L.B.C.~property for finite order elements of \mcg, we also
  show they act elliptically on \mg. More precisely,  

  \begin{introthm}\label{introthmPeriodic}

    Let $\mu_B \in \mcg$ be a fixed base point. There exist constants $R$
    and $k$ depending only on $\mu_B$ such that the following hold. For any
    finite order element $f \in \mcg$, there exists $\mu \in \mg$ such that
    $\mu$ is an $R$--fixed point of $f$ (i.e.\ $d_{\mg}(\mu,f\mu) \le R$)
    and 
    \begin{align} \label{eqPeriodic} 
       d_{\mg}(\mu_B, \mu) \le k |f|. 
    \end{align}

  \end{introthm}
  
  \begin{introcor}\label{introcorPeriodic}

    There exist a constant $K$, depending only on $\s$, and a finite set of
    elements $\Sigma \subset \mcg$ such that if $f \in \mcg$ has finite
    order, then there exists $\omega \in \mcg$ such that $\omega f
    \omega^{-1} \in \Sigma$ and $|\omega| \le K |f|$.

  \end{introcor}
  
  The proof of \thmref{introthmPeriodic} is the technical part of this
  paper. It is easy to see that there exists a constant $R_1$, depending
  only on \s, such that any finite order element $f \in \mcg$ acts on \mg
  with $R_1$--fixed points. The hard part is finding an $R_1$--fixed point
  of $f$ that lies ``$k$--close'' (in the sense of \eqref{eqPeriodic}) to
  $\mu_B$, for some uniform $k$. Using the projection maps from \mg to
  curve complexes, what we want is to find a marking $\mu \in \mg$ such
  that, for any $Z$, $\pi_Z(\mu)$ lies sufficiently close to the convex
  hull of $\big\{ \pi_Z (f^i\mu_B) \big\}$ (the projection to
  $\mathcal{C}(Z)$ of the orbit of $\mu_B$ under $\langle f \rangle$). To
  find such a $\mu$, our strategy is to take an arbitrary $R_1$--fixed
  point $\mu'$ of $f$ and construct from it a marking $\mu$ (possibly equal
  to $\mu'$) that satisfies \thmref{introthmPeriodic} for appropriate
  constants $k$ and $R$ ($R$ possibly bigger than $R_1$). The construction
  of $\mu$ is through a sequence of modifications on $\mu'$, taken place in
  subsurfaces of \s to which the projections of $\mu'$ is ``far'' from the
  convex hull. (If in every subsurface of \s, the projection of $\mu$ is
  not ``far'' from the convex hull, then $\mu=\mu'$.)  An important part of
  our proof that makes this process work is two technical lemmas
  (\secref{secTwoLemmas}), which show that the symmetries of the action of
  $f$ on \s can be detected by hierarchies. 

  To establish L.B.C.~property for reducible elements of \mcg, we combine
  the two arguments, for pseudo-Anosov elements and for finite order
  elements. If $f \in \mcg$ is a reducible element of infinite order, then
  up to taking powers the surface \s can be decomposed into a collection of
  subsurfaces on which $f$ is either pseudo-Anosov or has finite order. In
  order to apply induction to subsurfaces, we need to built paths from
  $\mu_B$ to $f\mu_B$ in \mg that move only in the complementary
  subsurfaces of the reducing system of $f$. This is possible if the
  initial marking $\mu_B$ contains the reducing system of $f$. However, one
  marking cannot contain all possible reducing systems, even up to
  conjugation. But it suffices to reduce to a finite problem. We show:

  \begin{introthm}\label{introthmReducible}

    There exist a constant $k$ and a finite set of markings $\mathcal{M}$ so
    that if $f \in \mcg$ is reducible, then there exists $\omega \in \mcg$ such
    that the reducing system of $\omega f \omega^{-1}$ is contained in some $\mu
    \in \mathcal{M}$ and $|\omega| \le k|f|$.

  \end{introthm}
   
  Finally, each case in the classification will produce a different
  constant. The proof of L.B.C.~property for \mcg will be completed by
  taking a maximum over the three constants.

  The organization of the paper is as follows.
  \begin{itemize}
    
    \item In \secref{secBackground}, we review basic definitions and the
    theory of hierarchies. A key notion that we will introduce is the
    notion of \emph{separating markings} in \secref{secSeparating}.

    \item In \secref{secTwoLemmas}, we give a couple of definitions and
    prove two technical lemmas about finite order mapping classes which
    will be useful for the next section. We also construct an example which
    motivates this section and the next section.
    
    \item In \secref{secFinite}, we prove \thmref{introthmPeriodic} and
    derive L.B.C.~property for finite order mapping classes.

    \item In \secref{secReducible}, we prove \thmref{introthmReducible} and
    use the known results for pseudo-Anosov and finite order elements to
    derive L.B.C.~property for infinite order reducible mapping classes.

  \end{itemize}
  
  \subsection{Acknowledgments} 
  
  The author would like to thank her thesis advisor, Howard Masur, for his
  excellent guidance and for suggesting the problem. She would also like to
  thank Kasra Rafi for helpful conversations throughout this project, and
  also Benson Farb, Daniel Groves, and Chris Leininger for helpful
  suggestions and their interest in this work. The author is also grateful
  to the referee for many helpful comments.

\section{Preliminaries}

  \label{secBackground}

  In this section, we develop the background material for the paper. Our
  main tool will be Masur and Minsky's theory of hierarchies. From
  \secref{secHierarchies} to \secref{secTimeOrder}, we will summarize the
  properties of hierarchies that will be needed for this paper. Some of the
  definitions will be merely sketched and most of the proofs will be
  omitted. We refer the reader to Masur-Minsky's paper \cite{MM00} for more
  details. We also refer to \cite{FLP79} and \cite{FM12} for general references
  on mapping class groups and the topology of surfaces, and to \cite{Gro87}
  and \cite{BH99} for references on $\delta$-hyperbolic spaces.

  \subsection{Arcs, curves, surfaces and subsurfaces} 
  
  Let $\s= \s_{g,p}$ be a connected, oriented surface of genus $g$ with $p$
  punctures. We call $\xi(\s) = 3g - 3 + p $ the \emph{complexity} of \s.
  Surfaces of complexity strictly greater than $1$ are called
  \emph{generic} surfaces. Surfaces of complexity $1$ are called
  \emph{sporadic} and they are topologically either the four-holed sphere
  or one-holed torus. Two remaining low-complexity cases are
  \emph{exceptional surfaces}. Complexity $0$ is the three-holed sphere or
  a \emph{pair of pants}, and complexity $-1$ is topologically an annulus.

  Throughout this paper we will be working with a generic surface $\s$
  without boundary. But sporadic and exceptional surfaces and surfaces with
  boundary naturally arise as subsurfaces of \s, and thus are important for
  induction arguments.  

  An \emph{essential curve} or just a \emph{curve} on \s will always mean
  the free isotopy class of a simple closed curve, which is not
  null-homotopic or homotopic to a puncture or a boundary component. A
  \emph{multicurve} or a \emph{curve system} will mean a finite collection
  of distinct curves that can be realized disjointly. A \emph{pants
  decomposition} of \s is a maximal curve system $c$ on \s. In particular,
  each component of $\s \setminus c$ is topologically a pair of pants. Note
  that a pants decomposition exists for \s if and only if $\xi(\s) \ge 1$,
  in which case the cardinality of $c$ is equal to $\xi(\s)$. 
  
  To talk about arcs we need \s to have boundary. An \emph{arc} on \s will
  be an isotopy class of a simple arc $\delta$, with isotopies relative to
  the boundary, such that $\delta$ has both endpoints on $\partial \s$ and
  is not isotopic to a boundary component. 
  
  The \emph{(geometric) intersection number} $i(\alpha, \beta)$ of a pair
  of curves $\alpha$ and $\beta$ will be the minimal number of
  intersections among representatives of $\alpha$ and $\beta$. The
  geometric intersection number between two arcs on \s will be the minimal
  number of intersections in the interior of \s modulo isotopies relative
  to $\partial \s$. Note that intersection number of an arc or curve with
  itself is always zero. 

  A \emph{subsurface} $Y$ of \s is the isotopy class of a closed and
  connected subsurface of \s which is incompressible and non-peripheral. We
  include the possibility that $Y = \s$ unless we say a \emph{proper}
  subsurface. By $\partial Y$ we will mean the multicurve comprised of the
  boundary components of a representative of $Y$. An annular subsurface $A$
  of \s is a regular neighborhood of a curve $\alpha$ with simple
  boundaries. We will often abuse terminology by confusing $A$ with its
  \emph{core curve} $\alpha$, and refer to $\alpha$ as a subsurface of \s
  as well. In this case, $\partial A$ will mean $\alpha$. For reasons we
  shall see, we will distinguish subsurfaces that are not pants, called
  \emph{essential subsurfaces} or \emph{domains}.

  Given a curve $\alpha$ and a non-annular domain $Y$ of \s, we will say
  $\alpha$ is disjoint from $Y$ if it can be homotoped away from a
  representative of $Y$. Note that this includes the case that $\alpha$ is
  a curve in $\partial Y$. If $\alpha$ can be realized as an essential
  curve in a representative of $Y$, then we will say $\alpha$ is a curve in
  $Y$. In all other cases, we will say $\alpha$ \emph{crosses} $Y$. For an
  annular domain $A$ with core curve $\beta$, then we have the
  possibilities that $\alpha$ is disjoint from $A$ if $\alpha$ and $\beta$
  are disjoint, or $\alpha$ crosses $A$ if $\alpha$ and $\beta$ intersect. 
  
  Similarly, given two domains $Y$ and $Z$ of \s, we will say $Y$ and $Z$
  are: \emph{disjoint} if $Y$ and $Z$ can be homotoped to be disjoint from
  each other; \emph{nested} if $Y$ and $Z$ can be homotoped so that either
  $Y$ is contained in $Z$ or $Z$ is contained in $Y$; and \emph{interlock}
  if they are neither disjoint or nested. Note that when $Y=A$ is an
  annular domain with core curve $\alpha$, then $A$ and $Z$ being disjoint
  is consistent with $\alpha$ and $Z$ being disjoint, $A$ and $Z$ are
  nested if $\alpha$ is contained $Z$ (and $Z$ is not annular), and,
  finally, $A$ and $Z$ interlock if $\alpha$ crosses $Z$.

  \subsection{Mapping class groups} 
  
  Let $\homeo^+(\s)$ be the group of orientation-preserving
  self-homeomorphisms of \s. The \emph{mapping class group} of \s is \[
  \mathcal{MCG}(\s) = \homeo^+(\s)/\sim \] where $f \sim g$ if and only if
  $g^{-1} \circ f$ is isotopic to the identity map on \s. Elements of \mcg
  are called \emph{mapping classes}. It is well-known that \mcg is finitely
  generated (and finitely presented) \cite{Lic64}. For this paper, we will
  fix a finite generating set $\Lambda$ of \mcg. We will often regard \mcg
  as a metric space by considering the word metric $|\cdot| =
  |\cdot|_\lambda$ induced by $\Lambda$. 
  
  If \s is a once-punctured torus or four-times punctured sphere, then \mcg
  is commensurable to $\Sl(2,\zz)$. The mapping class group of a
  thrice-punctured sphere is finite. For us, an annulus $A$ will always
  appear as a regular neighborhood of a simple closed curve on an ambient
  surface, so $A$ has two boundary components. Let $\mathcal{MCG}(A,
  \partial A)$ be the group of isotopy classes of homeomorphisms of $A$
  relative to $\partial A$. One checks that $\mathcal{MCG}(A, \partial A)$
  is homeomorphic to $\zz$.  
  
  \begin{definition}[L.B.C.~property] \label{defLBC}

  Given a finitely generated group $G$ equipped with a finite generating
  set $\Lambda$, we say a conjugacy class $\mathfrak{c}$ of $G$ has
  \emph{linearly bounded conjugators} if for any $f, g \in \mathfrak{c}$,
  there exists a conjugating element $\omega \in G$ such that \[ |\omega|
  \le K_\mathfrak{c} \big(|f| + |g|\big), \] where $|\cdot|$ represent the
  word length in $\Lambda$, and $K_\mathfrak{c}$ depends only on
  $\mathfrak{c}$ and $\Lambda$. If $K = K_\mathfrak{c}$ can be taken to be
  independent of the conjugacy class $\mathfrak{c}$, then we say $G$ has
  \emph{linearly bounded conjugator property} or \emph{L.B.C.~property}. If
  $G$ has L.B.C.~property for $\Lambda$, then changing $\Lambda$ to any
  other finite generating set changes $K$ by a bounded amount. Therefore,
  this definition is independent of the choice of the generating set, so
  $\Lambda$ can always be taken to be a symmetric generating set.

  \end{definition}
  
  Mapping class groups of non-generic surfaces satisfy L.B.C.~property. We
  would like to show the same is true for mapping class groups of generic
  surfaces. The first observation is that the Nielsen-Thurston
  classification of mapping classes is a conjugacy invariant. This means
  that we can argue for L.B.C.~property separately for each type. We refer
  to \cite{Thu88} and \cite[\S 13]{FM12} for more details on the
  classification theorem. Recall that a mapping class $f$ is called
  \emph{irreducible} if $f$ does not fix any multicurve (setwise);
  otherwise $f$ is called \emph{reducible}.  The following statement
  applies to all surfaces \s. 
  
  \begin{theorem}[Nielsen-Thurston classification for
  \mcg] \label{thmNTCMCG}
    
    Every element $f \in \mcg$ is either pseudo-Anosov, periodic (finite
    order), or reducible. Furthermore, for each $f \in \mcg$, there exists
    a (possibly empty) multicurve $\sigma$ invariant under $f$ with the
    following property. Let $Y_1, \ldots, Y_k$ be the connected components
    of $\s \setminus \sigma$, and, for each $i$, choose the smallest $n_i
    \in \nn$ so that $f^{n_i}(Y_i) = Y_i$. Then for any $i$,
    $f^{n_i}|_{Y_i}$ either has finite order or is pseudo-Anosov.

  \end{theorem}
  
  The multicurve $\sigma$ satisfying \thmref{thmNTCMCG} $f$ is called a
  \emph{reducing system} for $f$. For each $Y_i \in \s \setminus \sigma$,
  the map $f^{n_i}$ is called the \emph{first return map of $f$ to $Y_i$}.
  Note that the first return map of $f$ to $Y_i$ means exactly that
  $f^{n_i}|_{Y_i}$ can be viewed as an element of $\mathcal{MCG}(Y_i)$. The
  content of the classification can be rephrased to say $f$ is
  pseudo-Anosov if and only if $f$ is irreducible of infinite order.
  
  \begin{definition}[Canonical reducing systems]\label{secCanonical} 
    
    By choosing $\sigma$ to be a minimal collection of curves satisfying
    \thmref{thmNTCMCG}, then $\sigma = \sigma_f$ is unique up to isotopy
    and is called the \emph{canonical reducing system} for $f$. If $f \in
    \mcg$ is either pseudo-Anosov or finite order, then $\sigma_f =
    \emptyset$. (See \cite{Mos07}).
      
  \end{definition}
  
  In \cite[\S 7]{MM00}, Masur-Minsky established L.B.C.~property for the
  pseudo-Anosov elements of \mcg.

  \begin{theorem}[L.B.C.~property for pseudo-Anosov mapping
    classes]\label{thmPA}
   
    There exists a constant $K$, depending only on \s, such that if $f, g
    \in \mcg$ are conjugate pseudo-Anosov mapping classes, then there is a
    conjugating element $\omega \in \mcg$ with \[ |\omega| \le K \big( |f|
    + |g| \big). \]

  \end{theorem}
  
  Our goal in this paper is to prove L.B.C.~property for the finite order
  and reducible elements of \mcg. The argument for finite order mapping
  classes is the hard part of this paper. The argument for reducible
  mapping classes is inductive and will make use of the canonical reducing
  system. 
  
  \subsection{Complexes of curves} \label{secCC} 
  
  The \emph{complex of curves} \cc on a surface \s is a locally-infinite,
  finite dimensional simplicial complex on which \mcg acts by
  automorphisms. Its definition first appeared in \cite{Har81}. We treat
  generic, sporadic, and exceptional surfaces separately.  
  \begin{itemize}

    \item \textbf{Generic surfaces.} Suppose \s has $\xi(\s) > 1$. The
    $k$--th skeleton $\mathcal{C}_k(\s)$ consists of all curve systems on
    \s of cardinality $k+1$. There is an obvious inclusion of
    $\mathcal{C}_{k-1}(\s) \hookrightarrow \mathcal{C}_{k}(s)$ by face
    relations. Top dimensional simplices of $\mathcal{C}(\s)$ correspond to
    pants decompositions on \s, hence $\dim(\mathcal{C}(\s)) = \xi(\s) -
    1$.   
    
    \item \textbf{Sporadic surfaces.} With the above definition, the curve
    complex of a sporadic surface \s would be a disconnected set of points.
    To construct a more useful object, we modify the definition to allow
    two vertices in \cc span an edge if they intersect minimally over \s
    (once for one-holed torus and twice for four-holed sphere). It is a
    classical theorem that with this definition \cc is isomorphic to the
    Farey graph \cite{HT80,Min96b}.

    \item \textbf{Pants.} A pair of pants has no essential curves. Here we
    do not modify the definition and let the curve complex of pants be
    empty. This is the reason why we do not consider pants to be essential
    subsurfaces.

    \item \textbf{Annuli.} An arbitrary annulus has no essential curves.
    But for us, an annulus $A$ will always appear as a regular neighborhood
    of a curve $\gamma$ in a larger surface \s, and we would like
    $\mathcal{C}(A)$ (or $\mathcal{C}(\gamma))$ to record twist information
    about $\gamma$. Vertices of $\mathcal{C}(A)$ will be properly embedded
    arcs and two arcs are connected by an edge if they can be isotoped rel
    endpoints to have disjoint interiors. 
    
    \end{itemize}

  By an element or subset of \cc we will always mean an element or subset
  of $\mathcal{C}_0(\s)$.  We make \cc into a complete geodesic metric
  space by endowing each simplex with an Euclidean structure with edge
  lengths 1. From the perspective of coarse geometry, we do not lose
  anything by identifying \cc with its $1$-skeleton. We denote by
  $d_{\cc}$, or more simply by $d_\s$, the shortest distance in
  $\mathcal{C}_1(\s)$ between two vertices. If $A$ is an annulus with a
  core curve $\gamma$, we will also use the notation $d_\gamma$ or $d_A$ to
  denote distances in $\mathcal{C}(A)$. For any surface \s including
  annuli, induction on intersection number can be used to show \cc is
  connected, and $d_\s(\alpha, \beta) \le 2 i(\alpha,\beta)+1$ (see
  \cite{MM99,MM00}). The simplicial action of \mcg on \cc preserves this
  metric. The action is not proper. The quotient $\cc/\mcg$ parametrizes
  curves on \s up to homeomorphisms, hence it is finite. 
  
  For a generic surface \s, $d_\s$ coarsely measures the complexity between
  two curves in the following sense: $d_\s(\alpha,\beta) = 1$ if and only
  if $\alpha$ and $\beta$ are disjoint; $d_\s(\alpha,\beta) = 2$ if and
  only if $\alpha$ and $\beta$ cohabit a proper subsurface $Y \subset \s$;
  $d_\s(\alpha,\beta) \ge 3$ if and only $\alpha$ and $\beta$ \emph{fill}
  \s, or the complement of their union in \s does not support any essential
  curve.   
  
  The following theorem in \cite{MM99} gives us some geometric control over
  paths in \cc. 
  
  \begin{theorem}[{\cite{MM99}}]\label{thmHyperbolic}
        
    For any surface \s that is not a pair of pants, \cc has infinite
    diameter and is $\delta$--hyperbolic.

  \end{theorem}

  For sporadic surfaces, \thmref{thmHyperbolic} follows from a classical
  result that the Farey graph is quasi-isometric to an infinite-valence
  tree (see \cite{Man05}). In the case of an annulus $A$,
  \thmref{thmHyperbolic} follows from the fact that $\mathcal{C}(A)$ is
  quasi-isometric to $\zz$ (see \cite[\S 2.4]{MM00}).
  
  For generic surfaces, there are several ways to see that \cc has infinite
  diameter. Relevant to our paper is the following lemma.

  \begin{lemma}[{\cite[Proposition 4.6]{MM99}}] \label{lemPA}

    There exists $k = k(\s)$ such that for any pseudo-Anosov $f \in \mcg$,
    any vertex $v \in \cc$, and any $n \in \zz$, \[d_{\s} \big( v, f^n(v)
    \big) \ge k|n|. \]
  
  \end{lemma}
  
  The proof of $\delta$-hyperbolicity of \cc for a generic \s is
  nontrivial. We also refer to \cite{Bow06} for an alternate proof. 

  \subsection{Subsurface projections}\label{secProjection} 
  
  In this section, we restrict our discussion to domains of an ambient
  surface \s with $\xi(\s) \ge 1$. To do away with isotopy classes of
  curves and surfaces, we will equip \s with hyperbolic metric so that we
  may consider geodesic representatives for curves and (non-annular)
  subsurfaces of \s bounded by them.  
   
  Let $Y \subset \s$ be a proper domain. There is a map \[\pi_Y : \cc \to
  \mathcal{P}\big(\mathcal{C}(Y)\big),\] taking an element of \cc to a
  subset of $\mathcal{C}(Y)$ of bounded diameter. We call $\pi_Y(\alpha)$
  the \emph{projection} of $\alpha$ to $Y$. Note that in the definition
  below, the projection map also makes sense if we replace \s by any
  subsurface of \s that contains $Y$ as a proper subsurface. 

  We first define the projection to a non-annular domain $Y$. If $\alpha$
  and $Y$ are disjoint, then $\pi_Y(\alpha) = \emptyset$. If $\alpha$ is a
  curve in $Y$, then $\pi_Y(\alpha) = \{\alpha\}$. Otherwise, $\alpha$
  crosses $Y$ and $\alpha \cap Y$ consist of a collection of arcs in $Y$.
  The endpoints of each arc $\delta \subset \alpha \cap Y$ lie on one or
  two components of $\partial Y$. Let $N$ be a regular neighborhood of the
  union of $\delta$ with its corresponding component(s) in $\partial Y$.
  $N$ has either one or two components which are essential in $Y$. Let
  $\pi_Y(\delta)$ be the set of boundary component(s) of $N$. We define \[
    \pi_Y(\alpha) = \displaystyle \bigcup_{\delta \subset \alpha \cap Y}
  \pi_Y(\delta). \] Now suppose $Y=A$ is an annulus with core curve
  $\gamma$. There is a unique annular cover of $\s$ \[ p : \widehat{A} \to
  \s \] to which $A$ lifts homeomorphically. Since \s admits a hyperbolic
  metric, this cover has a natural compactification, also denote by
  $\widehat{A}$. We define $\mathcal{C}(A) = \mathcal{C}(\widehat{A})$. For
  any curve $\alpha$ in $\s$, components of $p^{-1}(\alpha)$ that are
  essential arcs form a subset in $\mathcal{C}(A)$. We will let
  $\pi_A(\alpha)$ be this corresponding set in $\mathcal{C}(A)$.  

  Denote by $\diam_Y(\cdot)$ the diameter of subsets in $\mathcal{C}(Y)$.
  For any two subsets $A, B \subset \mathcal{C}(Y)$, let \[ d_Y(A,B) =
  \diam_Y(A \cup B). \] Given a pair of curves $\alpha, \beta \in \cc$ and
  a domain $Y \subset \s$, we define \[ d_Y(\alpha, \beta) = d_Y \big(
  \pi_Y(\alpha), \pi_Y(\beta) \big). \] For any multicurve $\sigma$, one
  can also project $\sigma$ to $\mathcal{C}(Y)$ in the obvious way:
  $\pi_Y(\sigma) = \bigcup_{\alpha \in \sigma} \pi_Y(\alpha)$. Given two
  multicurves $\sigma$ and $\tau$, the distance $d_Y(\sigma, \tau)$ is
  similarly defined. 
  
  The follow result asserts that subsurface projections are coarsely
  well-defined and Lipschitz.

  \begin{lemma}[{\cite[Lemma 2.3]{MM00}}]\label{lem2lip}
    
    For any multicurve $\sigma$ on \s and any domain $Y \subset \s$, if
    $\pi_Y(\sigma) \ne \emptyset$, then $\diam_Y \big( \pi_Y(\sigma) \big)
    \le 2$. 
    %Thus, for any pair of simplices $\sigma$ and $\beta$ in $\cc$,
    %\[ d_Y(\sigma, \tau) \le 2 d_\s(\sigma, \tau). \]

  \end{lemma}

  Suppose $Y$ and $Z$ are domains of $\s$ such that $Y$ is contained in
  $Z$. Then the maps $\pi_Y$ and  $\pi_Y \circ \pi_Z$ are ``coarsely
  equal'' as maps from $\cc \to \mathcal{P}(\mathcal{C}(Y))$.

  \begin{lemma}[{\cite[Lemma 2.12]{BKMM06}}] \label{lem3lip}

   There exists a constant $M$ depending only on $\s$ such that for any
   multicurve $\sigma$, \[ \diam_Y \big(\pi_Y(\sigma), \pi_Y \circ
   \pi_Z(\sigma) \big) \le M. \]

  \end{lemma}

  We also have the following contraction property for the projection map from
  \cite[Theorem 3.1]{MM00}.

  \begin{theorem}[Bounded geodesic image]\label{thmBGI}

    There exists a constant $M_0$ depending only on $\s$ such that the
    following holds. Suppose $Y \subset \s$ is a proper essential
    subsurface, and $g$ is a geodesic in $\cc$ such that $\pi_Y(v) \ne
    \emptyset$ for every vertex $v \in g$. Then \[\diam_Y(g) \le M_0.\]

  \end{theorem}

  We say $g$ \emph{cuts} $Y$ if $\pi_Y(v) \ne \emptyset$ for every vertex
  $v \in g$, and $g$ \emph{misses} $Y$ otherwise. If $d_\s(g, \partial Y)
  \ge 2$ then $g$ cuts $Y$. On the other hand, by \thmref{thmBGI}, if $u, v
  \in \cc$ has $d_Y(u,v) > M_0$, then any geodesic $g$ in $\cc$ between $u$
  and $v$ misses $Y$.

  \subsection{Marking graph}
  
  \label{secMarking} 
  
  Another useful combinatorial object that admits an action by $\mcg$ is
  the \emph{marking graph} \mg of \s. Roughly, a \emph{marking} $\mu$ on \s
  is a multicurve $c$ on \s with additionally a set of transverse curves
  which serve to record twisting data about each curve in $c$. Below, we
  give a precise definition that works for any surface \s with $\xi(\s) \ge
  1$.
  
  A marking $\mu$ on \s is a set of ordered pairs $\{(\alpha_i,t_i)\}$,
  where the \emph{base curves} $base(\mu) = \{\alpha_i\}$ is a multicurve
  on \s, and each \emph{transversal} $t_i$ is either empty or is a
  diameter-1 set of vertices in $\mathcal{C}(\alpha_i)$. The set of
  transversals $\{t_i\}$ is denoted by $\tran(\mu)$. A transversal $t$ in
  the pair $(\alpha,t)$ is called \emph{clean} if $t =
  \pi_{\alpha}(\beta)$, where $\beta$ is a curve on $\s$ such that $\alpha$
  and $\beta$ are Farey-neighbors in the subsurface that they fill. A
  marking $\mu$ is \emph{clean} if every non-empty transversal $t$ is
  clean, and the curve $\beta$ inducing $t$ does not intersect any other
  base curve other than $\alpha$. A marking $\mu$ is called \emph{complete}
  if $\base(\mu)$ is a pants decomposition of $\s$ and no transversal is
  empty. If $\mu$ is complete and clean, then a transversal $t$ determines
  uniquely the curve $\beta$ such that $t = \pi_\alpha(\beta)$. If $\mu$ is
  not clean then there is bounded number ways of picking a
  \emph{compatible} clean marking $\mu'$, in the following sense:

  \begin{lemma}[{\cite[Lemma 2.4]{MM00}}]\label{lemCleanMarking}

    There exists a constant $M$ depending only on $\s$ satisfying the
    following. For any complete marking $\mu$ on $\s$, there exists a
    uniformly bounded number (depending only on \s) of complete clean
    markings $\mu'$ such that $\base(\mu) = \base(\mu')$, and
    $d_{\alpha}(t, t') \le M$ for any $(\alpha, t) \in \mu$ and $(\alpha,
    t') \in \mu'$.

  \end{lemma}
  
  We will often suppress the pair notation and regard a marking $\mu$ as
  the union of its base curves and transversals, i.e.~$\displaystyle \mu =
  \big( \cup_{\alpha \in \base(\mu)} \alpha\big) \bigcup \big( \cup_{t \in
  \tran(\mu)} t \big)$.

  \begin{definition}[Marking graph]
    
    The \emph{marking graph} $\mg$ is the graph with vertices representing
    complete clean markings on $\s$. Two vertices $\mu = \{ (\alpha_i,
    \pi_{\alpha_i}(\beta_i)) \}$ and $\mu' = \{ (\alpha'_i,
    \pi_{\alpha'}(\beta_i')) \}$ are connected by an edge if they differ by
    one of the following \emph{elementary moves}: 
    \begin{itemize}

      \item \emph{Twist}: For some $i$, $\beta_i'$ is obtained from $\beta$ by a
        twist or half-twist along $\alpha_i$. All base curves and other
        transversals of $\mu$ and $\mu'$ agree.
        
      \item \emph{Flip}: Let $\mu''$ be the (unclean) marking obtained from
        $\mu$ by ``flipping'' $ \big( \alpha_i, \pi_{\alpha_i}(\beta_i)
        \big)$ to $\big( \beta_i, \pi_{\beta_i}(\alpha_i) \big)$, for some
        $i$. The marking
        $\mu'$ is any clean marking compatible with $\mu''$ replacing
        all transversals $\beta_j$ that intersect $\beta_i$.   
    
    \end{itemize}

  \end{definition}

  We equip \mg with the combinatorial edge metric, denoted by $d_{\mg}$.
  Like \cc, \mg is connected and admits an action of \mcg by isometries.
  But unlike \cc, \mg is locally finite and the action of \mcg is proper.
  The quotient $\mg/\mcg$ is also finite, since there are only finitely
  many complete clean markings up to homeomorphisms of $\s$ \cite{MM00}. By
  a standard application of \v{S}varc-Milnor, the orbit map $\mcg \to \mg$
  is a quasi-isometry.     

  \begin{definition}[Projection of markings]

  Let $Y \subset \s$ be essential and let $\mu \in \mg$. We can project
  $\mu$ to $Y$, also denoted by $\pi_Y(\mu)$, in the following way. Namely,
  if $Y$ is not a curve in $\base(\mu)$, then $\pi_Y(\mu) = \pi_Y \big(
  \base(\mu) \big)$. If $Y=\alpha$ is a curve contained in $\base(\mu)$,
  then $\pi_Y(\mu) = t$, where $t$ is the transversal curve to $\alpha$ in
  $\mu$. Note that, since $\mu$ is a complete marking, the projection map
  is always non-empty. 
  
  Since $\base(\mu)$ is a diameter-1 set in $\cc$, in light of
  \lemref{lem2lip} the projection map is Lipschitz: 
  
  \end{definition}

  \begin{lemma}[{\cite[Lemma 2.5]{MM00}}]\label{lem4lip}
    
    For any $\mu, \nu \in \mg$ and any domain $Y \subseteq \s$, \[ d_Y(\mu,
    \nu) \le 4 \, d_{\mg}(\mu, \nu). \]

  \end{lemma}
  
  If $c$ is a multicurve and $\mu$ a marking with $c \subseteq \base(\mu)$,
  then we say $\mu$ is an \emph{extension} of $c$. We will often start with
  a multicurve $c$ and \emph{extend} it to a marking $\mu$. This amounts to
  choosing a marking on all the essential non-annular components of $\s
  \setminus c$ and choosing a transversal for each curve $\alpha \in c$.
  There are many ways to extend a marking in general, but most often we
  will need the marking $\mu$ to satisfy certain desired properties so
  those choices will be bounded. 

  \begin{definition}[Induced Marking]

  Let $Y \subset \s$ be an non-annular domain. We define a map \[ \Pi_Y :
  \mg \to \operatorname{Mark}(Y). \] For each marking $\mu$ on $\s$, choose
  a pants decomposition $b$ of $Y$ such that $b$ has minimal intersection
  with $\pi_Y(\mu)$. We extend $b$ to a marking $\nu = \Pi_Y(\mu)$ on $Y$
  as follows. For each curve $\alpha \in b$, choose transversal $t_\alpha$
  in $Y$ such that $d_\alpha \big(t_\alpha, \mu \big)$ is minimal. The
  marking $\nu=\{(\alpha,t_\alpha): \alpha \in b\}$ will be called an
  \emph{induced marking} of $\mu$ on $Y$, and it is well-defined up to a
  bounded number of choices. It follows from \lemref{lem3lip} and
  \lemref{lem4lip} that for any marking $\mu$, any non-annular domain $Y
  \subset \s$, and any domain $Z \subset Y$,
  \begin{equation}\label{eqInduced} d_Z \big(\mu, \Pi_Y(\mu) \big) \le M,
  \end{equation} where $M$ depends only on \s.
  
  \end{definition}

  \begin{definition}[Relative marking extension]\label{defMarkingExtension}

    Let $\mu \in \mg$ and $c$ be a multicurve on \s. We extend $c$ to a
    marking $\mu' \in \mg$ \emph{relative} to $\mu$ as follows. For each
    non-annular domain $Y$ in $\s \setminus c$, choose an induced marking
    $\Pi_Y(\mu)$ on $Y$. Then for each curve $\alpha \in c$, choose a
    transversal $t_\alpha$ with minimal $d_\alpha(t_\alpha, \mu)$. The
    union of $\{(\alpha,t_\alpha): \alpha \in c\}$ with the set of induced
    markings $\Pi_Y(\mu)$ forms a marking $\mu' \in \mg$ which is
    well-defined up to a bounded number of choices.

  \end{definition}
  
  The following is an immediate consequence of our construction.   

  \begin{lemma} \label{lemMarkingExtension}

    Let $c$ be a multicurve on \s, $\mu$ any marking, and $\mu'$ an
    extension of $c$ relative to $\mu$. For any proper domain $Z \subset
    \s$, if $Z$ is contained in an essential component of $S \setminus c$,
    or if $Z$ is a curve in $c$, then \[ d_Z(\mu', \mu) \le M, \] where $M$
    depends only on \s.

  \end{lemma}
   
  \begin{proof}
    
    For any curve $\alpha$ in $c$, the transversal $t_\alpha$ to $\alpha$
    in $\mu'$ was chosen to be uniformly close to $\pi_\alpha(\mu)$. Thus
    $d_\alpha(\mu', \mu)$ is uniformly bounded by a constant depending on
    \s. Now suppose $Z \subseteq Y$ where $Y$ is a component of $\s
    \setminus c$. By construction $\pi_Y(\mu') = \base \big( \Pi_Y(\mu)
    \big)$. Thus, by \eqref{eqInduced}, $d_Z(\mu',\mu) = d_Z \big(
    \Pi_Y(\mu), \mu \big)$ is also uniformly bounded by a constant depending
    only on \s. \qedhere

  \end{proof}

  \subsection{Hierarchies}
  
  \label{secHierarchies} 
  
  In the previous section, we introduced the marking graph \mg which is
  quasi-isometric to \mcg. In this section, we will introduce the theory of
  hierarchies, which is useful for constructing efficient paths in \mg.
  These paths are naturally associated to efficient representations of
  elements in \mcg in terms of the generators, thus justifying \mg as a
  good combinatorial model for \mcg. 
  
  The idea of hierarchies is to associate to every pair of markings a
  family of geodesics in curve complexes that behave well with subsurface
  projections. In order for the theory to work, we need to impose a
  condition on geodesics in curve complexes called \emph{tightness}. Let $Y
  \subseteq \s$ be a domain. A \emph{tight geodesic} $g$ in
  $\mathcal{C}(Y)$ is a sequence $\{ v_0, \ldots, v_n \}$ of simplices in
  $\mathcal{C}(Y)$, such that any sequence of vertices in $g$ is a geodesic
  in $\mathcal{C}(Y)$ in the usual sense, and $v_{i-1} \cup v_{i+1}$ fill a
  subsurface $Z \subset Y$ such that $\partial Z = v_i$. We remark that the
  original definition \cite[Definition 4.2]{MM00} consists of more
  information. 
  
  It is a theorem of Masur-Minsky that any two points in $\mathcal{C}(Y)$
  is connected by at least one and at most finitely many tight geodesics
  \cite[Lemma 4.5 and Corollary 6.14]{MM00}. Henceforth, a geodesic in a
  curve complex will always mean a tight geodesic. By an abuse of notation,
  we will refer to $v_i$'s as vertices of $g$. We will say the
  \emph{length} of $g$ is $n$, and write $|g| = n$. We will say $Y$ is the
  \emph{domain} or \emph{support} of $g$, and write $D(g) = Y$. We will
  sometimes use the notation $[v_0,v_n]$ to mean any geodesic from $v_0$ to
  $v_n$ in $\mathcal{C}(Y)$. Since $\mathcal{C}(Y)$ is $\delta$-hyperbolic,
  all (finitely many) geodesics from $v_0$ to $v_n$ are fellow-travelers.
  
  We now briefly sketch the definition of a hierarchy. For a complete
  definition, see \cite[Definition 4.4]{MM00}. A \emph{hierarchy} on \s is a
  collection $H$ of geodesics such that each geodesic $g \in H$ is
  supported on some domain $Y \subseteq \s$, with a distinguished
  \emph{main geodesic} $g_H=[v_0,v_n]$ supported on $\s$, together with
  some additional structure and satisfying certain conditions which we now
  highlight. A hierarchy $H$ comes equipped with a pair of markings $I(H)$
  and $T(H)$ on \s, called the \emph{initial marking} and the
  \emph{terminal marking} of $H$, respectively, such that $v_0 \subseteq
  \base(I(H))$ and $v_n \subseteq \base(T(H))$. We will usually assume
  $I(H)$ and $T(H)$ are complete clean marking on $\s$. One of the key
  technical conditions of a hierarchy is called \emph{subordinacy}.
  Roughly, given a geodesic $g$ in \cc, one can inductively construct a
  hierarchy $H$ with $g=g_H$. For each vertex $v_i$ in $g_H$, the vertices
  $v_{i-1}$ to $v_{i+1}$ are contained in some component $Z$ of $\s
  \setminus v_i$. The geodesic $h=[v_{i-1},v_{i+1}]$ in $\mathcal{C}(Z)$
  will be an element of $H$ and is subordinate to $g_H$. One can continue
  this process with all vertices of $g_H$ and then with $h$ and so on. 
  
  We list some properties of hierarchies below, after the following
  definition. 
 
  \begin{definition}[Component domain]

     Given a non-annular domain $Y \subseteq \s$, and a multicurve $c$ on
     $Y$, we say $Z$ is a \emph{component domain} of $(Y, c)$ if $Z$ is
     either an essential component of $Y \setminus c$ or $Z$ is a curve in
     $c$. 
  
  \end{definition}
  
  \begin{theorem} The following statements hold for hierarchies.

     \begin{itemize}
       
        \item[1](Existence) Given any markings $\mu$ and $\nu$ on $\s$,
        there exists a hierarchy $H$ with $I(H) = \mu$ and $T(H) = \nu$
        \cite[Theorem 4.6]{MM00}. 
       
        \item[2](Uniqueness of geodesics) For any hierarchy $H$, if $h, h'
        \in H$ have $D(h) = D(h')$, then $h = h'$ \cite[Theorem 4.7]{MM00}.
       
        \item[3] (Completeness) For every geodesic $h \in H$ and vertex $v
        \in h$, if $Y$ is a component domain of $\big( D(h),v \big)$, then
        $Y$ is domain for a geodesic $k \in H$ \cite[Theorem 4.20]{MM00}.
   
     \end{itemize}
   
  \end{theorem}
  
  We will sometimes denote an hierarchy from $\mu$ to $\nu$ by
  $H(\mu,\nu)$.   
  The following lemma explains the relationship between a geodesic $h \in
  H$ and the projection of $I(H)$ and $T(H)$ to $D(h)$. 
  
  \begin{lemma}[{\cite[Lemma 6.2]{MM00}}]\label{lemLargeLink}
    
    There exist constants $M_1 > M_2$, depending only on $\s$, such that if
    $H$ is any hierarchy in $\s$ and \[ d_Y \big(I(H),T(H) \big) \ge M_2 \]
    for a subsurface $Y$ in $\s$, then $Y$ is a domain for a geodesic $h
    \in H$.

    Conversely, if $h \in H$ is any geodesic with $Y = D(h)$, then $h$
    fellow travels any geodesic from $\pi_Y \big( I(H) \big)$ to $\pi_Y
    \big( T(H) \big)$ in $\mathcal{C}(Y)$ with a uniform constant. In
    particular, 
    \[ 
       \big| |h| - d_Y\big(I(H), T(H)\big) \big| \le M_1 .
    \]

  \end{lemma}
  
  For any pair of markings $\mu, \nu \in \mg$, we will call a domain $Y$ a
  \emph{large link} for $\mu$ and $\nu$ if $d_Y(\mu, \nu) \ge M_2$.

  The theorem below summarizes two results that are vital to this paper. To
  simplify the statements we introduce some notations that we will adopt
  for the rest of the paper. Below, $\mathfrak{a}$ and $\mathfrak{b}$
  represent quantities such as distances or lengths, and $k$ and $c$ are
  constants that depend only on \s (unless otherwise noted). 
  
  \begin{notation}
  
    \begin{enumerate}
    
    \item If $\mathfrak{a} \le k \mathfrak{b}+c$, we say $a$ is
    \emph{coarsely bounded} by $b$, and write $\mathfrak{a} \prec
    \mathfrak{b}$

    \item If $\dfrac{1}{k} \mathfrak{b} - c \le \mathfrak{a} \le k
    \mathfrak{b} + c$, we say $a$ is \emph{coarsely equal} to $b$, and
    write $\mathfrak{a} \asymp \mathfrak{b}.$

  \end{enumerate}
  \end{notation}
  
  By the \emph{length} $|H|$ of a hierarchy $H$ we will mean $|H| = \sum_{h
  \in H} |h|$. In the following, the coarse equality on the left is
  \cite[Theorem 6.10]{MM00}. The coarse equality on the right is called the
  distance formula \cite[Theorem 6.12]{MM00}.

  \begin{theorem} \label{thmDF}

     There exists a constant $L_0$ depending only on $S$ such that, for any
     $L \ge L_0$ and any $\mu, \nu \in \mg$ and any hierarchy
     $H=H(\mu,\nu)$, 
     \[
       |H| \quad \asymp \quad d_{\mg}(\mu,\nu) \quad \asymp \sum_{\stackrel{Y
        \subseteq \s}{d_Y(\mu,\nu)\ge L}} d_Y(\mu,\nu).
     \]
     On the right, the constants involved in $\asymp$ depend on $L$.
  \end{theorem}

  Fix a generating set $\Lambda$ for $\mcg$. To realize the quasi-isometry
  between $\mg$ and $\mcg$, we fix a base marking $\mu_B$ in $\mg$. Then
  $d_{\mg}(\mu_B, f\mu_B) \asymp |f|$, with constants depending only on $\mu_B$
  and $\Lambda$. The following is an immediate consequence of
  \thmref{thmDF}.

  \begin{corollary}[{\cite[Theorem 7.1]{MM00}}]
    
     Let $\mu_B \in \mg$ be a fixed base marking. For any element $f \in
     \mcg$,
     \[
       |H(\mu_B,f\mu_B)| \quad \asymp \quad |f|.
     \]

  \end{corollary}
  
  Let $Y \subset \s$ be a proper non-annular domain. There is a coarse
  embedding $\Mark(Y) \stackrel{j}{\longrightarrow} \mg$ obtained as
  follows. Fix a marking in each essential component of $\s \setminus Y$
  and a transversal to each curve in $\partial Y$. The map $j$ sends the
  marking $\nu \in \Mark(Y)$ in an obvious way so that $\partial Y
  \subseteq \base\big(j(\nu)\big)$ and for all $\nu_1, \nu_2 \in \Mark(Y)$,
  \begin{equation} \label{eqUndistorted} d_{\Mark(Y)}(\nu_1, \nu_2) \asymp
    d_{\mg} \big( j(\nu_1), j(\nu_2) \big). 
  \end{equation}
  Equation \eqref{eqUndistorted} follows from the distance formula and one
  can make the coarse constants independent of $Y$ and $j$. 
  
  \subsection{Slices} \label{secSlices}

  The connection between paths in \mg and hierarchies come from
  \emph{slices} of a hierarchy. The following definition comes from
  \cite[\S 5]{MM00}.

  \begin{definition}[Slices]

    A \emph{(complete) slice} of a hierarchy $H$ is a set $\tau$ of
    \emph{pointed geodesics} $(h,v)$ in $H$, i.e.~$h \in H$ and $v$ is a
    vertex of $h$, satisfying the following properties: 

    \begin{itemize}
      
      \item[(S1)] Any geodesic $h$ of $H$ appears at most once in $\tau$.

      \item[(S2)] There is a distinguished pair, the \emph{bottom pair}, $(g_H,
        b)$ of $\tau$.    

      \item[(S3)] For every $(k, w) \in \tau$ other than the bottom pair, $D(k)$
        is a component domain of $\big( D(h), v \big)$ for some $(h,v) \in \tau$.

      \item[(S4)] Given $(h,v) \in \tau$, for every component domain $Y$ of
        $\big( D(h), v \big)$ there is a pair $(k,w) \in \tau$ with $D(k) = Y$. 

    \end{itemize}

  \end{definition}

  The \emph{initial slice} $\tau_0$ of $H$ is one where every pair $(h,v)
  \in \tau$ has $v$ the first vertex of $h$. In particular, the main
  geodesic $g_H$ and its initial vertex is a pair in $\tau_0$, and $\tau_0$
  can be constructed inductively using the axioms of slices. Similarly, the
  \emph{terminal slice} of $H$ is defined. 
  
  To any slice $\tau$ we can associate a complete marking $\mu_\tau$ as
  follows. First, let $\mu$ be the marking with \[ \base(\mu) = \{\,\, v
  \,\,\,:\,\,\, (h,v) \in \tau \text{ and } D(h) \text{ is not an
  annulus}\,\, \}. \] For each base curve $\alpha$, if $(k,t) \in \tau$ is
  such that $k$ is a geodesic in $\mathcal{C}(\alpha)$, then let $t$ be the
  transversal to $\alpha$ in $\mu$. The marking $\mu$ is complete but not
  necessarily clean. Any clean marking $\mu_\tau$ compatible with $\mu$
  will be called a \emph{compatible} marking with $\tau$. By
  \lemref{lemCleanMarking}, the number of choices for $\mu_\tau$ is
  bounded. Note that $I(H)$ and $T(H)$ are respectively compatible markings
  with the initial and terminal slice of $H$. We will call any marking
  compatible with some slice in a hierarchy $H$ a \emph{hierarchal
  marking} of $H$. 
  
  Given any slice $\tau$ in $H$, there is a notion of \emph{(forward)
  elementary move} on $\tau$ which is roughly moving a vertex $v$ of some
  pair $(h,v) \in \tau$ forward by one step in the geodesic $h$ to obtain a
  new slice $\tau'$. We write $\tau \to \tau'$. (See \cite[\S 5]{MM00} for a
  precise definition.) If $\mu$ and $\mu'$ are compatible marking with
  $\tau$ and $\tau'$, then by \cite[Lemma 5.5]{MM00}, $d_{\mg}(\mu,\mu')
  \prec 1$. We will write $\mu \to \mu'$ to mean any path in \mg connecting
  $\mu$ to $\mu'$. To prove $|H| \asymp d_{\mg}(I(H),T(H))$, Masur and
  Minsky in \cite{MM00} established the existence of a \emph{resolution} of
  $H$, which is a sequence of forward elementary moves \[ \tau_0 \to \cdots
  \to \tau_n, \] where $\tau_0$ is the initial slice and $\tau_n$ is the
  terminal slice of $H$. For each $\tau_i$ in the resolution, let $\mu_i$
  be a compatible marking with $\tau_i$. The corresponding path in \mg \[
  I(H) = \mu_0 \to \cdots \to \mu_n = T(H)\] is a quasi-geodesic with
  uniform constants, and $d_{\mg}\big(I(H),T(H)\big) \asymp n \asymp |H|$.
  A fact in \cite{Min10} that we will sometime need is that, for any slice
  $\tau$ in $H$, there is a resolution of $H$ containing $\tau$. 
  
  The following statements are true for hierarchal markings and follow from
  \cite{MM00}.

  \begin{lemma} \label{lemTriangle}
    
    Let $H$ be a hierarchy. If $\mu \in \mg$ is a hierarchal marking of
    $H$, then
    \begin{align} \label{eqMultiplicative}
      d_{\mg} \big(I(H), \mu\big) + d_{\mg} \big(\mu, T(H)\big) \prec
      d_{\mg} \big(I(H), T(H)\big)
    \end{align}
    There exists a constant $M$ depending only on \s such that for any
    domain $Y \subseteq \s$,
    \begin{align} \label{eqAddictive}
      d_Y\big(I(H), \mu\big) + d_Y\big(\mu, T(H)\big) \le  d_Y\big(I(H),
      T(H)\big) + M
    \end{align}
    
    \begin{proof}
       
      Let $g$ be a quasi-geodesic in \mg containing $\mu$ coming from a
      resolution of $H$. By \cite{MM00}, $g$ is a quasi-geodesic from $I(H)$
      to $T(H)$ with uniform constants, hence \eqref{eqMultiplicative}
      holds.
      
      Given $Y \subseteq \s$, let $\pi_Y(g)$ be the projection of $g$ to
      $\mathcal{C}(Y)$ (project each vertex of $g$ to $\mathcal{C}(Y))$.
      The projection $\pi_Y(g)$ is a quasi-geodesic in $\mathcal{C}(Y)$
      with uniform constant. By hyperbolicity of $\mathcal{C}(Y)$,
      $\pi_Y(g)$ stays uniformly close to any geodesic connecting $\pi_Y
      \big( I(H) \big)$ and $\pi_Y \big( T(H) \big)$ of $\pi_Y(g)$. Thus
      there exists a constant $M_Y$ such that
      \begin{align*}
         d_Y\big(I(H), \mu\big) + d_Y\big(\mu, T(H)\big) \le  d_Y\big(I(H),
         T(H)\big) + M_Y.
      \end{align*}
      Since there are only finitely many subsurfaces of \s up to
      homeomorphism, the constant $M = \max_Y \{M_Y\}$ depends only on \s
      and achieves \eqref{eqAddictive}. \qedhere 
       
    \end{proof}

  \end{lemma}
  
  \subsection{Time order} \label{secTimeOrder}

  The geodesics or domains of geodesics in a hierarchy $H$ satisfy a
  partial order $<_t$, called \emph{time order}. We refer to \cite[\S
  4.6]{MM00} for the definition. The idea comes from the observation that
  the vertices of a geodesic $g$ are linearly ordered: $v_i<v_j$ if $i<j$.
  Combining this observation with the subordinacy structure on $H$, one can
  try to order a pair of geodesics $h, h' \in H$. 
  %A simple situation is if $D(h)$ is a component domain of $(D(g),v_i)$
  %and $D(h')$ is a component domain of $(D(g),v_j)$, then $D(h) <_t D(h')$
  %if $i < j$. 
  In the following, we summarize some main results and state some useful
  consequences of time order.
  
  %If $D(h)$ and $D(h')$ are different components of $(D(g),v_i)$, then in
  %fact they are not time-ordered. The situation gets more complicated if
  %$h$ and $h'$ are not subordinate to the same geodesic $g$. An important
  %fact is that if $D(h)$ and $D(h')$ interlock, then they must be
  %time-ordered. If $D(h)$ and $D(h')$ are nested, then they cannot be
  %time-ordered. 
  %

  \begin{theorem}[{\cite[Lemma 4.18 and 4.19]{MM00}}] \label{thmTimeOrder}
    
     There exists a relation $<_t$, called \emph{time-order}, on domains of
     geodesics in $H$ such that:
     \begin{itemize}
      
        \item The relation $<_t$ is a strict partial order.

        \item If $h$ and $h'$ are geodesics in $H$ such that $Y = D(h)$ and
        $Z = D(h')$ interlock, then either $Y <_t Z$ or $Z <_t Y$.

        \item If $Y \subset Z$, then $Y$ and $Z$ are not time-ordered.

        \item If $Y$ and $Z$ lie in different component
        domains of $ \big( D(m),v \big)$, for some geodesic $m$ in $H$ and
        $v \in m$, then $Y$ and $Z$ are not time-ordered.
        
    \end{itemize}

  \end{theorem}
  
  Note that the ambiguous case is when $D(h)$ and $D(h')$ are disjoint;
  sometimes they are time-ordered and sometimes not. The issue of disjoint
  domains will come up in this paper. 
  
  The constant $M_1$ of \lemref{lemLargeLink} can be chosen so that
  following hold. 

  \begin{lemma}[{\cite[Lemma 6.11]{MM00}}]\label{lemTimeOrder}
    
    Let $H$ be a hierarchy. Suppose $Y$ and $Z$ are domains for geodesics
    in $H$ such that $Y$ and $Z$ interlock. If $Y <_t Z$, then
    $d_Y\big(\partial Z, T(H)\big) \le M_1$ and $d_Z\big(I(H), \partial
    Y\big) \le M_1$.

  \end{lemma}
  
  Using slices, the constant $M_1$ can be chosen so the following version
  of \lemref{lemTimeOrder} also holds.

  \begin{lemma} \label{lemTimeOrder1}
    
    With the same hypothesis as above. There exists a hierarchy marking
    $\nu$ such that $d_Y\big(\nu, T(H)\big) \le M_1$ and $d_Z(I(H), \nu)
    \le M_1$.

  \end{lemma}
  
  \begin{proof}
     
     By assumption, both $Y$ and $Z$ are domains for geodesics for a
     hierarchy $H$ with $Y <_t Z$. Let $k \in H$ be the geodesic supported
     on $Y$ and let $w \in k$  be the terminal vertex of $k$. Let $\tau$ be
     a slice $\tau$ with $(k,w) \in \tau$ (such $\tau$ exists by
     \cite[Lemma 5.8]{Min10}). Let $\nu$ a hierarchal marking compatible
     with $\tau$. Since $k$ is supported on $Y$, by definition of a slice,
     there exists some pair $(h,u) \in \tau$ such that $Y$ is a component
     domain of $\big( D(h),v \big)$. By definition of a compatible marking,
     we have $\partial Y \subseteq \base(\nu)$, which implies that, by
     \lemref{lemTimeOrder}, $d_Z( I(H), \nu)$ is uniformly bounded. Since
     $w$ is the terminal vertex of $k$, any resolution of $H$ containing
     $\tau$ does not pass through $Y$ from $\tau$ to $T(H)$. Thus, $d_Y
     \big( \nu, T(H) \big)$ is also uniformly bounded. This finishes the
     proof of the lemma. \qedhere 

  \end{proof}

  We may choose $M_1$ so the following also holds: 

  \begin{lemma}[{\cite[Lemma 1]{BM08}}]\label{lemTimeOrder2}
    
    With the same hypothesis as above. For any marking $\mu \in \mg$,
    either $d_Y\big(\mu, T(H)\big) \le 2M_1$ or $d_Z(I(H),\mu) \le 2M_1$.

  \end{lemma}
  
  \subsection{Separating Marking} \label{secSeparating}

  The following definition and lemma do not explicitly appear in
  \cite{MM00}. Although the lemma is a direct consequence of hierarchies, we
  offer a brief sketch of its proof.

  \begin{definition}[Separating marking]

     Let $H$ be a hierarchy. A slice $\tau$ is called a \emph{separating
     slice} if for every pair $(h,v) \in \tau$, with $h \ne g_H$, has the
     property that $v$ is the terminal vertex of $h$. We remark that once
     the bottom pair $(g_H,b)$ is fixed, then the separating slice $\tau$
     containing $(g_H,b)$ is uniquely determined by the axioms of slices.
     In particular, if $b$ is the terminal vertex of $g_H$, then $\tau$ is
     the terminal slice of $H$. If $\tau$ is a separating slice containing
     $(g_H,b)$, then any marking $\mu$ compatible with $\tau$ is called a
     \emph{separating marking} at $b$.  
  
  \end{definition}
  
  The constant $M_1$ of \lemref{lemLargeLink} can be chosen so that the
  following hold. 

  \begin{lemma} \label{lemSeparatingMarking}

    Let $H$ be a hierarchy. Let $b$ be any vertex in $g_H$ and let $\mu$ be
    a separating marking at $b$. Then for any proper domain $Y \subset \s$,
    either $d_Y \big( I(H),\mu \big) \le M_1$ or $d_Y \big( \mu,T(H) \big)
    \le M_1$.
  
  \end{lemma}

  \begin{proof}

    We may assume $Y \subset \s$ has $d_Y\big(I(H), T(H)\big) > M_1$. Since
    $M_1 \ge M_2$, $Y$ is a domain for a geodesic $h_Y \in H$. Without a
    loss of generality, we may assume $Y$ is a component domain of
    $(g_H,c)$, for some $c$ in $g_H$. If $c$ appears before $b$ along
    $g_H$, then $d_Y\big(\mu, T(H)\big) \le M_1$. Similarly, if $c$ appears
    after $b$ along $g_H$, then $d_Y(I(H), \mu) \le M_1$. Both of these
    facts can be seen as a consequence of \lemref{lemTimeOrder1}. The
    remaining case is $b = c$. In this case, the separating slice
    containing $(g_H,b)$ must contain $(h_Y,v)$, where $v$ is the terminal
    vertex of $h_Y$. Therefore, it must be that $d_Y\big(\mu, T(H)\big) \le
    M_1$.
  \end{proof}

  \begin{remark}

     In our definition of separating slice, the preference for terminal
     vertices is arbitrary. \lemref{lemSeparatingMarking} would remain true
     if we allowed only initial vertices or a mixture of initial and
     terminal.
     
  \end{remark}
  
  \subsection{Collecting constants} \label{secConstants}

  For the rest of the paper, we will fix the following set of constants.

  Let $M_0$ be the constant of \thmref{thmBGI}. Let $L_0$ be the constant
  of \thmref{thmDF}. Let $M_1$ and $M_2$ be the constants coming from
  \lemref{lemLargeLink}. We will also fix one constant $M_3$ for
  \lemref{lem3lip}, \lemref{lemCleanMarking}, Equation \eqref{eqInduced},
  \lemref{lemMarkingExtension}, and \lemref{lemTriangle}. We may assume
  $M_1 \ge M_2, M_3$. In addition, since up to homeomorphism there are only
  finitely many subsurfaces of \s, we can choose a hyperbolicity constant
  $\delta$ which works for all $\mathcal{C}(Z)$, $Z \subseteq \s$. 
  
  %Finally,
  %we will let $N$ be the constant of \corref{lemOrderBound} below. 
  
\section{Two technical lemmas}
  
  \label{secTwoLemmas} 
  
  This section contains some technical results about finite order mapping
  classes. 
  
  To prove L.B.C.~property, we need to understand the geometry of the
  action of finite order mapping classes on \mg. The first observation is
  that finite order elements act on \mg with coarse fixed points. We will
  eventually prove that the action has the property that the translation
  distance of a finite order element $f$, or $d_{\mg}(\mu_B, f\mu_B)$ where
  $\mu_B$ is the base marking, is coarsely bounded by the distance from
  $\mu_B$ to the fixed point sets of $f$. In other words, finite order
  elements of \mcg act \emph{elliptically} on \mg.
  
  In this section, we consider what happens if a fixed point $\mu$ of $f$
  is far from $\mu_B$ relative to the translation distance of $f$. By the
  distance formula, there must be some $X \subseteq S$ such that
  $d_X(\mu_B,\mu)$ is large relative to $d_X(\mu_B,f\mu_B)$. With some
  additional conditions, $X$ will be called a \emph{bad domain} for $\mu$
  and we will prove a structure theorem for the set of bad domains in a
  hierarchy $H(\mu_B,\mu)$. In the next section, we will use this structure
  theorem to construct a coarse fixed point of $f$ close to $\mu_B$
  relative to the translation distance of $f$. From there, we can derive
  L.B.C.~property for finite order mapping classes by a standard argument.
  
  \subsection{Fixed points and symmetric points}

  We state some useful facts about finite order mapping classes below. 

  \begin{lemma} \label{lemFiniteConj}

    There are finitely many conjugacy classes of finite order elements in
    \mcg.

  \end{lemma}
  
  \begin{corollary}\label{lemOrderBound}

    There exists a constant $N$, depending only on \s, such any finite
    order element $f \in \mcg$ has $\ord (f) \le N$.

  \end{corollary}

  \begin{definition}
    
    Let $f \in \mcg$ be of finite order. We define the set of \emph{$r$-fixed
    points} of $f$ as 
    \[ 
      \fix_r(f) = \{\,\, \mu \in \mg \,\,\, : \,\,\, d_{\mg}(\mu,f\mu) \le r
       \,\,\}. 
    \]  
    Also, define the set of \emph{$r$-symmetric points} for $f$ to be 
    \[
       \wfix_r(f) = \{ \,\, \mu \in \mg \,\,\, :
       \,\,\, d_Y( \mu, f\mu) \le r, \, \forall\, Y \subseteq \s \,\,\}.
    \] 

  \end{definition}

  \begin{lemma}\label{lemFixedPoints}
     
    There exists a constants $R_1$ depending only on $\s$ such that
    $\fix_{R_1}(f) \ne \emptyset$ and $\wfix_{R_1}(f) \ne \emptyset$, for
    any finite order element $f \in \mcg$.

  \end{lemma}

  \begin{proof}
    
    Choose any $\mu \in \mg$ and let $R_f = d_{\mg}(\mu, f\mu)$. If $g = \omega
    f \omega^{-1}$ for some $\omega \in \mcg$, then $d_{\mg}(\omega \mu,
    g\omega \mu) = d_{\mg}(\mu, f\mu) \le R_f$. Thus $\fix_{R_f}(g) \ne
    \emptyset$ for all $g$ in the conjugacy class of $f$. Using
    \lemref{lemFiniteConj}, we can let $R_1$ be the maximum of the
    constants $R_f$ ranging over all conjugacy classes. This gives
    $\fix_{R_1}(f) \ne \emptyset$ for all finite order element $f \in
    \mcg$. Using the distance formula, we can choose $R_1$ so
    that $\wfix_{R_1}(f) \ne \emptyset$ as well.
  \end{proof}
  
  Henceforth, we will fix $R_1$ to be the minimal constant satisfying
  \lemref{lemFixedPoints}.

  \begin{remark} \label{remFixedPoint} 
  
    We can describe the geometry of the subset $\fix_{R_1}(f) \subset \mg$.
    By Nielsen Realization, any finite order element $f \in \mcg$ can be
    realized as an isometry of a hyperbolic surface $X$ \cite{Ker83}. The
    quotient $\bar{X} = X/f$ is an orbifold. One can coarsely identify
    $\fix_R(f)$ with $\operatorname{Mark}(\bar{X})$. The map $X \to
    \bar{X}$ is a (branched) covering map. By \cite{RS09}, the lifting of
    $\operatorname{Mark}(\bar{X})$ to $\operatorname{Mark}(X)=\mg$ is a
    quasi-isometric embedding.
  
  \end{remark}
  
  \subsection{An example}

  Before proceeding to the first technical lemma, let's discuss a
  motivating example. The following refers to \figref{figFixedPoint}.

  Consider the closed surface $\s_2$ of genus two. Let $f$ be the mapping
  class of order two which permutes the two holes of $S_2$. Let $\alpha$ be
  the separating curve in \s indicated in \figref{figFixedPoint}. Let $X$
  and $Y$ be the pair of once-punctured tori in $S_2$ with boundary
  $\partial X = \alpha = \partial Y$ The map $f$ permutes $X$ and $Y$.
  Using this and the fact that $X$ and $Y$ are disjoint, we can construct a
  family of coarse fixed points of $f$ as follows (see
  \remref{remFixedPoint}). Since $X$ is a once-punctured torus,
  $\mathcal{C}(X)$ is homeomorphic to the Farey graph $\mathcal{F}$. Via
  the map $f$, we can also identify $\mathcal{C}(Y)$ with $\mathcal{F}$.
  After this identification, markings on $X$ or $Y$ correspond to edges in
  $\mathcal{F}$. Choose any marking $\nu$ on $X$. By the action of $f$, we
  get a marking $f(\nu)$ on $Y$, which is represented by the same edge in
  $\mathcal{F}$. Choose a transverse curve $\beta$ to $\alpha$ so that $\mu
  = \nu \cup f(\nu) \cup \alpha \cup \beta$ is a clean marking. Since
  $f(\mu) = \mu$, $\mu$ is a fixed point of $f$.

  See \figref{figFixedPoint} for a concrete example of a fixed point $\mu$
  of $f$ from this construction, where we have color-coded the curves so
  that the red or vertical curves represent the base curves of a marking
  $\mu$ and the blue or horizontal curves represent the transversal curves
  of $\mu$. In the example, let $\nu_1$ be the marking on $X$ obtained by
  the $(0,1)$ and $(1,0)$ curves. %Set $\nu_2 = f(\nu_1)$. 
  The marking $\mu = \nu \cup f(\nu) \cup \alpha \cup \beta$ is a
  $0$--fixed point of $f$. 
  
  \begin{figure}[htbp]

    \def\hole#1{%
    \begin{scope}[shift={#1}, scale=.8]%
      \draw (-.705,.077) arc (200:340: .75 and .5);%
      \draw (.564,-.077)  arc (20:160: .6 and .4);%
    \end{scope}%
    }

    \begin{center}

      \begin{tikzpicture}[thick]

        \coordinate                (a1) at (45: 1.5 and 1);
        \coordinate[shift={(3,0)}] (b1) at (135: 1.5 and 1);
        
        \coordinate[shift={(3,0)}] (a2) at (45: 1.5 and 1);
        \coordinate[shift={(6,0)}] (b2) at (135: 1.5 and 1);
        
        \coordinate[shift={(6,0)}] (a3) at (45: 1.5 and 1);
        \coordinate[shift={(9,0)}] (b3) at (135: 1.5 and 1);
        
        \coordinate                (c1) at (-45: 1.5 and 1);
        \coordinate[shift={(3,0)}] (d1) at (-135: 1.5 and 1);
        
        \coordinate[shift={(3,0)}] (c2) at (-45: 1.5 and 1);
        \coordinate[shift={(6,0)}] (d2) at (-135: 1.5 and 1);
        
        \coordinate[shift={(6,0)}] (c3) at (-45: 1.5 and 1);
        \coordinate[shift={(9,0)}] (d3) at (-135: 1.5 and 1);
        
        \draw (45:1.5 and 1) arc (45:315: 1.5 and 1);
        \hole{(0,0)}

        \draw (a1)       .. controls +(.3,-.2) and +(-.1,0)   .. (1.5,0.52);
        \draw (1.5,0.52) .. controls +(.1,0)   and +(-.3,-.2) .. (b1);
        
        \draw (c1)        .. controls +(.3,.2) and +(-.1,0)  .. (1.5,-0.52);
        \draw (1.5,-0.52) .. controls +(.1,0)  and +(-.3,.2) .. (d1);

        \draw[shift={(3,0)}] (135: 1.5 and 1) arc (135:-135: 1.5 and 1);
        \hole{(3,0)}

        \draw[blue] (0,0) ellipse (20pt and 10pt);
        \draw[blue] (3,0) ellipse (20pt and 10pt);
        \draw[blue] (.5,0) .. controls +(.3,.3) and +(-.3,.3) ..(2.5,0);

        \draw[purple] 
        (1.5,0.52) .. controls (1.6,0.3) and (1.6,-0.3) .. (1.5,-0.52);
        \draw[purple, dotted] 
        (1.5,0.52) .. controls (1.4,0.3) and (1.4,-0.3) .. (1.5,-0.52);

        \draw[purple, shift={(0,0)}]
        (0,-1) .. controls (0.1,-.745) and (0.1,-.445) .. (0,-.19);
        \draw[purple, dotted, shift={(0,0)}]
        (0,-1) .. controls (-0.1,-.745) and (-0.1,-.445) .. (0,-.19);

        \draw[purple, shift={(3,0)}]
        (0,-1) .. controls (0.1,-.745) and (0.1,-.445) .. (0,-.19);
        \draw[purple, dotted, shift={(3,0)}]
        (0,-1) .. controls (-0.1,-.745) and (-0.1,-.445) .. (0,-.19);
        
        \draw[->] (0,1.2) .. controls +(.5,.5) and +(-.5,.5) .. 
        node[midway,above] {$f$} (3,1.2);
        
        \node at (-1.8,0) {$X$};
        \node at (4.8,0) {$Y$};
        \node at (1.5,.7) {$\alpha$};
        \node at (2,.35)  {$\beta$};
        \node at (1.5,-1.5) {$\mu$};
        \node at (-.3,-.6) {$\nu$};
        \node at (3.5,-.6) {$f(\nu)$};
        
        \begin{scope}[shift={(8,0)}]
        
           \node at (0,2.3) {$\mathcal{F}$};
           \node at (-.3,.3) {$\nu$};
           \node at (-.4,-.3) {$f(\nu)$};
           \node at (-.3,0) {$\shortparallel$};
           \draw (0,0) circle (2cm);
           \draw (0,-2) -- (0,2);
           \filldraw (24:2) circle (.07);
           \filldraw (47:2) circle (.07);
           \filldraw (343:2) circle (.07);
           \filldraw (314:2) circle (.07);
           \node at (33 :1.45) {$z_1$};
           \node at (332:1.25) {$z_2$};
           %\node at (60 :1.55cm) {$\nu_1$};

        \end{scope}

        \begin{scope}[shift={(8,0)}]
           \clip (0,0) circle (2cm); 
           \draw (335:2.40cm) arc (0:360:.5cm);
           \draw (30:2.35cm) arc (0:360:.4cm);
           %\draw (52:2.15cm) arc (0:360:.25cm);
        \end{scope}

        \filldraw[purple] (8,2) circle (.07);
        \filldraw[blue] (8,-2) circle (.07);

      \end{tikzpicture}

    \end{center}

    \caption{On the left, the curves represent a $0$--fixed point $\mu$ for
      the order two element $f \in \mathcal{MCG}(\s_2)$ that permutes the
      holes of $\s_2$. On the right is the Farey graph $\mathcal{F}$ which
      is isomorphic to $\mathcal{C}(X)$ and $\mathcal{C}(Y)$;
      $\mathcal{C}(X)$ and $\mathcal{C}(Y)$ are identified via $f$.
      Markings on $X$ or $Y$ correspond to edges in $\mathcal{F}$.} Here,
      $\Pi_X(\mu) = \nu$ and $\Pi_Y(\mu) = f(\nu)$ represent the same edge
      in $\mathcal{F}$. The base marking $\mu_B$ (which is not drawn on the
      left) has $\Pi_X(\mu_B) = z_1$ and $\Pi_Y(\mu_B) = z_2$.

    \label{figFixedPoint} 

  \end{figure}
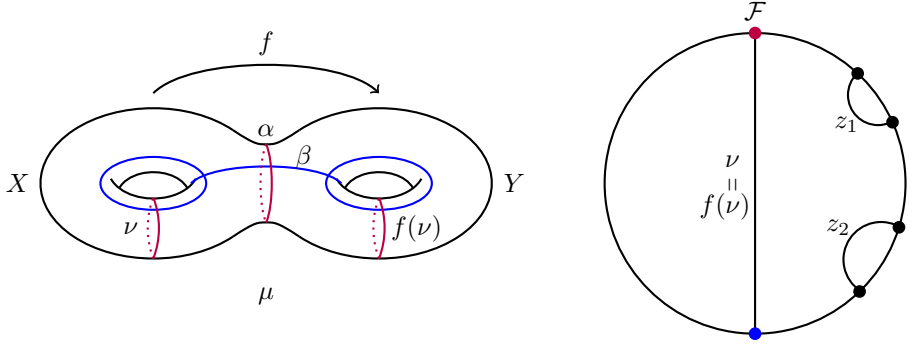

  Consider a base marking $\mu_B$ constructed as follows. For simplicity,
  we will assume $\alpha$ is a base curve of $\mu_B$. Choose two edges
  $z_1$ and $z_2$ in $\mathcal{F}$ that are very far apart. We will let
  $z_1$ be the marking in $X$ and $z_2$ be the marking in $Y$. Now choose a
  transverse curve $\beta'$ to $\alpha$ so that $\mu_B = z_1 \cup z_2 \cup
  \alpha \cup \beta'$ is clean. Since $z_1$ and $z_2$ are far, $\mu_B$ is
  itself not a coarse fixed point of $f$. Let $\mathcal{Z}_B$ be the
  collection of domains for which $d_Z(\mu_B,f\mu_B) \ge L_0$, where $L_0$
  is the constant of \thmref{thmDF}. Since $\alpha$ is a base curve of
  $\mu_B$, if $Z \in \mathcal{Z}_B$ then either $Z \subseteq X$ or $Z
  \subseteq Y$. Also, since $d_Z(\mu_B, f\mu_B) = d_{f(Z)}(\mu_B, f\mu_B)$,
  if $Z \in \mathcal{Z}_B$ then $f(Z) \in \mathcal{Z}_B$. Finally, since
  $d_X(\mu_B, f\mu_B) = d_\mathcal{F}(z_1,z_2)$ is large, $X$ (and $Y$) is
  in $\mathcal{Z}_B$. By \thmref{thmDF},
  \[
     d_{\mg} (\mu_B, f\mu_B) \asymp \sum_{Z \in \mathcal{Z}_B} d_Z(\mu_B,
     f\mu_B). 
  \] 
  
  To find a fixed point $\mu$ of $f$ ``close'' to $\mu_B$, 
  \begin{align} \label{eqClose} 
     d_{\mg}(\mu_B, \mu) \prec d_{\mg}(\mu_B, f\mu_B),
  \end{align}
  consider the following construction. Let $g$ be a geodesic in
  $\mathcal{F}$ connecting $\base(z_1)$ and $\base(z_2)$ (note that the
  convex hull of $\base(z_1)$ and $\base(z_2)$ is a finite set of
  geodesics). Let $\nu$ be any edge in $g$, which we will regard as a
  marking in $X$. Let $\beta$ be a transverse curve to $\alpha$ so that
  $d_\alpha(\beta, \beta')$ is uniformly bounded and $\mu = \nu \cup f(\nu)
  \cup \alpha \cup \beta$ is clean (\lemref{lemCleanMarking}). We show
  $\mu$ is ``close'' to $\mu_B$. By assumption, $d_\alpha(\mu_B,\mu)$ is
  uniformly bounded. Since $\alpha$ is contained in both $\mu_B$ and $\mu$,
  if $Z$ is any domain which contains $\alpha$ or is crossed by $\alpha$,
  we have $d_Z(\mu_B, \mu) \le 4$. For any $Z \subseteq X$, we have
  (ignoring some addictive errors) \[ d_Z (\mu_B, \mu) = d_Z(z_1, \nu) \le
  d_Z \big( z_1, f(z_2) \big) = d_Z(\mu_B, f\mu_B).\] Similarly, for any
  $Z \subseteq Y$, 
  \[ 
    d_Z(\mu_B, \mu) = d_Z \big(z_2, f(\nu) \big) \le d_Z \big(z_2, f(z_1)
    \big) = d_Z(\mu_B, f\mu_B). 
  \] 
  Let $\mathcal{Z}_{\mu}$ be the set of domains for which $d_Z(\mu_B,
  \mu) \ge L_0$. By the above computations, if $Z \in \mathcal{Z}_\mu$, then $Z
  \in \mathcal{Z}_B$. Thus we have 
  \begin{align*}
    d_{\mg}(\mu_B, \mu_1) 
    &\asymp \sum_{Z \in \mathcal{Z}_\mu} d_Z(\mu_B, \mu_1) \\
    & \le \sum_{Z \in \mathcal{Z}_\mu} d_Z(\mu_B, f\mu_B)\\
    & \le \sum_{Z \in \mathcal{Z}_B} d_Z(\mu_B, f\mu_B) \\
    & \asymp d_{\mg}(\mu_B, f\mu_B)
  \end{align*}
  Hence $\mu$ satisfies \eqref{eqClose}. We emphasize that, by varying the
  choice of the edge in $g$, we obtain a family of fixed points of $f$
  ``close'' to $\mu_B$. 

  In the following and in the subsequent section, we generalize this
  example. The general situations could be much more complicated; for
  instance, the assumption that $\mu_B$ contains $\alpha$ simplified the
  example quite a bit. The reason why our construction worked is, because
  in every domain $Z \subseteq \s$, $d_Z(\mu_B, \mu) \le
  d_Z(\mu_B,f\mu_B)$. Thus, if a coarse fixed point $\mu$ is not ``close''
  to $\mu_B$, then there should be some $Z \subseteq \s$ for which
  $d_Z(\mu_B, \mu) > d_Z(\mu_B, f\mu_B)$. This is the motivation behind
  \defref{defBadDomains} of a bad domain $Z$ for $\mu$ (the actual
  definition contains a slightly stronger condition.) In our construction
  of coarse fixed points, we relied heavily on the fact that $X$ and $Y$
  are disjoint and $X = f(Y)$. In general, we will also try to look for a
  domain $X$ such that $\{ f^i(X)\}$ are all pairwise disjoint. The
  structure result for bad domains, \lemref{lemBadDomains}, shows that if
  $X$ is a bad domain for $\mu$, then $X$ and its orbits are all disjoint
  and are all (essentially) bad domains for $\mu$. In \secref{secFinite},
  we will show that, when a coarse fixed point $\mu$ of $f$ does not have
  any bad domains, then $\mu$ will be the desired marking ``close'' to
  $\mu_B$ (\propref{propEmpty}). Otherwise, we will show how to use the
  disjointness result of a bad domain $X$ and its orbits to construct a new
  coarse fixed point of $f$ ``closer'' to $\mu_B$ (see
  \secref{secBaseStep}). Finitely many iteration of this construction will
  lead to a desired coarse fixed point of $f$ ``close'' to $\mu_B$ (see
  \secref{secIteration}). 
  
  \subsection{Bad domains and the first technical lemma}
  
  \label{secBadDomains} 

  We remark that \lemref{lemFiniteConj} does not a priori help us with
  L.B.C.~property as each conjugacy class has infinitely many elements, but
  it will play an essential role later. 
  
  In the following, let $\mu_B \in \mg$ be the fixed base marking. We
  recall notations of \secref{secConstants} and let $R_1$ be the minimal
  constant satisfying \lemref{lemFixedPoints}. Set 
  \begin{equation} \label{eqTheta} 
     \Theta = 6M_1 + 4\delta. 
  \end{equation}
  
  Let $f \in \mcg$ be of finite order. Fix $N$ to be the smallest constant
  satisfying \corref{lemOrderBound}. For any proper domain $X \subset \s$,
  let $L_X=L_X(f)$ be the integer such that $f^{L_X + 1}$ is the first
  return map of $f$ to $X$. Note that for any $X$, $L_X < \ord(f) \le N$.

  \begin{definition}[Bad domains]\label{defBadDomains} Let $R\ge R_1$ and
  let $\mu \in \mg$ be any marking. We say a domain $X \subseteq \s$ is a
  \emph{$R$-bad domain} for $\mu$ (and $f$) if
  \begin{itemize} 
    
    \item For $X = \s$, 
      \begin{equation}\label{eqBD1}
        d_\s(\mu_B, \mu) > d_\s(\mu_B,f\mu_B) + R.
      \end{equation}
  
    \item For $X \ne \s$,  
      \begin{equation}\label{eqBD2}
        d_X(\mu_B, \mu) > 2N \Big( \max_{0 \le i \le L_X} \big\{
        d_{f^i(X)}(\mu_B,f\mu_B) \big\} \Big) + NR + \Theta.
      \end{equation}

  \end{itemize} 

  \end{definition}

  Denote by $\Omega(\mu, R, f)$ or $\Omega(\mu, R)$ the set of all $R$-bad
  domains for $\mu$ (and $f$). Note that if $R' > R$, then $ \Omega(\mu,
  R') \subseteq \Omega(\mu, R).$ We remark that the constant ``$2N$'' in
  the definition of bad domains will not play a role until the next
  section. 

  \begin{definition}[Partial order on $\Omega(\mu,R)$]
    
    We endow $\Omega(\mu, R)$ with a  partial order ``$\vartriangleleft$''
    following these rules. Let $X, Y \in \Omega(\mu, R)$, and let $H =
    H(\mu_B, \mu)$ be a fixed hierarchy.
    \begin{enumerate} 
      
      \item[(O1)] If $\xi(X) < \xi(Y)$, then $X \vartriangleleft Y$. In
        particular, if $\s \in \Omega(\mu, R)$, then $\s$ is the maximal
        element.  
    
      \item[(O2)] If $\xi(X) = \xi(Y)$ and $X <_t Y$ in $H(\mu_B,\mu)$,
      then $Y \vartriangleleft X$.

    \end{enumerate}

  \end{definition}

  \begin{definition}[Complexity of $\Omega(\mu, R)$] \label{defComplexity}

    If $\Omega(\mu,R)$ is nonempty, then the complexity of the maximal
    element in $\Omega(\mu,R)$ is called the \emph{complexity of
    $\Omega(\mu, R)$}, denoted by $\xi(\mu,R)$. The minimal complexity over
    all subsurfaces of $\s$ is $-1$, coming from an annulus. For
    convenience, we will let $\xi(\emptyset) = -2$. We make the trivial
    observations that whenever 
    \[ 
       \Omega(\mu', R') \subseteq \Omega(\mu, R) \qquad \Longrightarrow \qquad
    \xi(\mu', R') \le \xi(\mu, R). 
    \]
   
  \end{definition}

  The following is a consequence of the definition of $R$-bad domains for
  $\mu$ and $f$, if $\mu$ happens to be a $R$--symmetric point for $f$. 
  
  \begin{lemma}[Structure of bad domains]\label{lemBadDomains}

    Let $R \ge R_1$ and let $\mu \in \wfix_R(f)$. If $X \in \Omega(\mu, R)$
    and $X \ne \s$, then 
    \[ 
       X, f(X), \ldots f^{L_X}(X) 
    \] 
    are all domains for geodesics in $H(\mu_B, \mu)$ and are all pairwise
    disjoint.

  \end{lemma}

  \begin{proof}
   
    Note that $f^{-n}(X) = f^{L_X+1-n}(X)$. We will prove, by induction on
    $n$, that \[ X, f^{-1}(X), \ldots, f^{-n}(X) = f^{L_X+1-n}(X)\] satisfy
    the conclusion of the lemma for $n = 0, \ldots, L_X$.  Our assumption
    is that $X \in \Omega(\mu,R)$ and $X \ne \s$. In particular, this means
    $d_X(\mu_B, \mu) > \Theta > M_2$, so $X$ is a domain for a geodesic in
    $H(\mu_B, \mu)$. This concludes the base case. 
    
    Let's now assume 
    \[ 
       X, f^{-1}(X), \ldots, f^{-n+1}(X) 
    \] 
    are all domains for geodesics in $H(\mu_B, \mu)$ and are all pairwise
    disjoint. We will show $f^{-n}(X)$ supports a geodesic in $H(\mu_B,
    \mu)$. Recursively, we have
    \begin{align*}\label{eqBD3}
      d_{f^{-n}(X)}&(\mu_B, \mu) \\ 
      & = d_{f^{-n+1}(X)}(f\mu_B, f\mu) \\
      & \ge d_{f^{-n+1}(X)}(\mu_B,\mu) - d_{f^{-n+1}(X)}(\mu_B, f\mu_B) -
      d_{f^{-n+1}(X)}(\mu,f\mu) \\ 
      & > d_{f^{-n+2}(X)}(f\mu_B,f\mu) - d_{f^{-n+1}(X)}(\mu_B, f\mu_B) - R \\
      & \hspace{1cm} \vdots \\
      & \ge d_X(\mu_B, \mu) - \left( \sum_{i=1}^n d_{f^{-n+i}(X)}(\mu_B,
      f\mu_B) \right) - nR \\
      & \ge d_X(\mu_B, \mu) - n \Big( \max_{1 \le i \le n} \big\{
      d_{f^{-n+i}(X)}(\mu_B, f\mu_B) \big\} \Big) - nR \\
      \text{By } \eqref{eqBD2} \quad
      & > (N-n) \Big( \max_{1 \le i \le L_X+1} \big\{ d_{f^{-n+i}(X)}(\mu_B,
      f\mu_B) \big\} \Big) + (N-n)R + \Theta 
    \end{align*}
    Since $N > L_X \ge n$, we have in particular
    \[ 
       d_{f^{-n}(X)}(\mu_B, \mu) > \Theta > M_2. 
    \] 
    Therefore, $f^{-n}(X)$ supports a geodesic in $H(\mu_B,\mu)$. 

   Now let's prove $f^{-n}(X)$ is disjoint from each $X, \ldots,
   f^{-n+1}(X)$. Observe that $f^{-n}(X)$ and $f^{-i}(X)$ are disjoint if
   and only if $f^{-n+i}(X)$ and $X$ are disjoint. Hence it is enough to
   show $f^{-n}(X)$ and $X$ are disjoint. By way of contradiction, let's
   assume $X$ and $f^{-n}(X)$ are not disjoint. The two domains have the
   same complexity so they must interlock. They both support geodesics in
   $H(\mu_B,\mu)$ so, by \thmref{thmTimeOrder}, they are time-ordered. We
   have two cases.

    The first case is $X <_t f^{-n}(X)$. As in \lemref{lemTimeOrder1}, we
    may choose a hierarchal marking $\nu$ for $H(\mu_B,\mu)$ such that 
    \begin{align} \label{eqTimed}
       d_X(\nu, \mu) \le M_1 \quad \text{and} \quad d_{f^{-n}(X)}(\mu_B,
       \nu) \le M_1, 
    \end{align}
    By the triangle inequality, 
    \begin{align} \label{eqBD3} 
       d_X(\mu_B, f^n\nu) 
       & \le d_X(\mu_B, f^n\mu_B) + d_X(f^n\mu_B, f^n\nu) \nonumber \\
       & \le \left( \sum_{j=0}^{n-1} d_X(f^j\mu_B, f^{j+1}\mu_B) \right) +
       d_{f^{-n}(X)}(\mu_B, \nu) \nonumber \\ 
       & = \left( \sum_{j=0}^{n-1} d_{f^{-j}(X)}(\mu_B, f\mu_B) \right) +
       d_{f^{-n}(X)}(\mu_B, \nu) \nonumber \\ 
       & \le n \left( \max_{0 \le j \le n-1} \left\{\, d_{f^{-j}(X)}
       (\mu_B,f\mu_B) \,\right\} \right) + M_1.
    \end{align} 
    Using the triangle inequality again, along with \eqref{eqBD2} and
    \eqref{eqBD3}, we have 
    \[ 
      d_X(f^n\nu, \mu) \ge d_X(\mu_B,\mu) - d_X(\mu_B, f^n\nu) > 2M_1. 
    \]
    Therefore, by \lemref{lemTimeOrder2}, 
    \[ 
       d_{f^{-n}(X)}(\mu_B, f^n\nu) \le 2M_1. 
    \] 
    Now we consider $d_X(\mu_B, f^{in}\mu)$. By iterating the argument we
    obtain inductively, for every $i \ge 0$,
    \begin{align} \label{eqBD4}
       d_X(\mu_B, f^{(i+1)n} \nu) 
       & \le  d_X(\mu_B, f^n\mu_B) + d_X(f^n\mu_B, f^{(i+1)n}\nu) \nonumber \\ 
       & \le n \Big( \max_{0 \le j \le n-1} \big\{\, d_{f^{-j}(X)}(\mu_B,
       f\mu_B) \,\big\} \Big) + d_{f^{-n}(X)}(\mu_B, f^{in}\nu) \nonumber \\ 
       & \le n \Big( \max_{0 \le j \le n-1} \big \{\, d_{f^{-j}(X)}
       (\mu_B,f\mu_B) \,\big\} \Big) + 2M_1.
    \end{align} 
    Since there is some $i$ for which $f^{in}$ is the identity map, 
    inequality \eqref{eqBD4} must also hold for $d_X(\mu_B, \nu)$. With
    this fact coupled with the first half of \eqref{eqTimed}, we derive the
    following violation to $X \in \Omega(\mu,R)$: 
    \begin{align*} 
       d_X(\mu_B, \mu) & \le d_X(\mu_B, \nu) + d_X(\nu, \mu) \\ 
       & \le n \Big( \max_{0 \le j \le L_X} \big\{ d_{f^{-j}(X)}(\mu_B,
       f\mu_B) \big\} \Big) + 2M_1 + d_X(\nu,\mu) \\
       & \le N \Big( \max_{0 \le j \le L_X} \big\{ d_{f^{-j}(X)}(\mu_B,
       f\mu_B) \big\} \Big) + 3M_1  
    \end{align*}
    Thus, it is not possible for $X <_t f^{-n}(X)$.
    
    To eliminate the second case $f^{-n}(X) <_t X$, we argue similarly. Now
    choose a marking $\nu$ such that 
    \begin{align} \label{eqOtherTimed}
       d_X(\mu_B, \nu) \le M_1 \quad \text{and} \quad d_{f^{-n}(X)}(\nu,
       \mu) \le M_1. 
    \end{align}
    Then, using the fact that $\mu \in \wfix_R(f)$
    (in the last step below), we have 
    \begin{align*} 
       d_X(f^n\nu, \mu) & \le d_X(f^n\nu, f^n\mu) + d_X(\mu, f^n\mu) \\
       & \le d_{f^{-n}(X)}(\nu, \mu) + \sum_{i=0}^{n-1}
       d_{f^{-i}(X)}(\mu, f\mu) \\ 
       & \le M_1 + nR. 
    \end{align*}
    Thus, \[ d_X(\mu_B, f^n\nu) \ge d_X(\mu_B, \mu) - d_X(f^n\nu, \mu) > 2M_1.
    \] By \lemref{lemTimeOrder2}, 
    \[
       d_{f^{-n}(X)}(f^n \nu, \mu) \le 2M_1. 
    \] 
    As above, we iterate the argument on taking powers of $f^n$. For all $i
    \ge 0$, we have 
    \begin{align*}
       d_X(f^{(i+1)n}\nu, \mu) & \le d_X(f^{(i+1)n}\nu, f^n\mu) + d_X(\mu,
       f^n\mu) \\ 
       & \le d_{f^{-n}(X)}(f^{in}\nu, \mu) + nR \\ 
       & \le 2M_1 + nR.
    \end{align*}
    This eventually leads to the contradiction 
    \begin{align*}
       d_X(\mu_B, \mu) 
       & \le d_X(\mu_B, \nu) + d_X(\nu, \mu) \\
       & \le M_1 + 2M_1 + nR \\
       & \le 3M_1 + nR.
    \end{align*}
    We conclude $X$ and $f^{-n}(X)$ must be disjoint.
  \end{proof}

  \subsection{Second technical lemma}

  In the following, we prove another technical result, which has a similar
  conclusion as \lemref{lemBadDomains}, but it is based on different
  assumptions. Its purpose is for the situation when the main surface \s is
  a bad domain for an $R$--fixed point $\mu$ of a finite order mapping
  class $f$ (see \propref{propBaseStep1}). In this situation, we need to
  cook up a set of base curves for a new fixed point of $f$ which is closer
  to the base marking $\mu_B$ in \cc. \lemref{lemFiniteOrder} starts this
  process by finding a subsurface whose orbit under $f$ are all pairwise
  disjoint. Furthermore, the boundary curves of these subsurfaces form a
  multicurve which is closer to $\mu_B$ in \cc than the base curves of
  $\mu$. 
  
  The proof of \lemref{lemFiniteOrder} will be similar to that of
  \lemref{lemBadDomains}. We will provide the details of the first half of
  the proof to illustrate the differences and omit the second half.

  In the following, let $R$ be any constant and let $\mu \in \wfix_R(f)$
  for a finite order mapping class $f$. Let $H(\mu_B,f\mu_B)$ and
  $H(\mu_B,\mu)$ be hierarchies. Let $[v_B, f(v_B)]$ and $[v_B, v]$ be the
  corresponding main geodesics in \cc. For any domain $Y \subset \s$, we
  adopt the notation $[u,v]_Y$ to mean the line segment $\big[
  \pi_Y(u),\pi_Y(v) \big]$ in $\mathcal{C}(Y)$.  For any two sets $A , B
  \subset \cc$, let $\displaystyle \dist_\s(A,B) = \min_{\stackrel{u\in
  A}{\scriptscriptstyle v\in B}} d_S(u,v)$.

  \begin{lemma}\label{lemFiniteOrder}
   
    There exists a constant $\Delta$, depending only on $R$, such that the
    following holds. Suppose $b$ is a vertex on $[v_B, v]$ with the
    property that \[ \dist_\s \Big( \big[ b, f(b) \big], \big[ v_B, f(v_B)
    \big] \Big) \ge 4 .\] Let $\mu'$ be a separating marking at $b$. Then
    whenever a subsurface $Y \subset \s$ has the property that \[ \dist_\s
    \Big( \partial Y, \big[ b,f(b) \big] \Big) \le 1 \quad \text{and} \quad
    d_Y(\mu',f\mu') > \Delta, \] then \[ Y, f(Y), \ldots, f^{L_Y}(Y)\] are
    all domains for geodesics in $H(\mu_B,\mu)$ and are all mutually
    disjoint.
    
  \end{lemma}
  
  \begin{proof}

    We claim the constant \[ \Delta = (2N+1) M_0 + (2N+1)R + 10M_1 \]
    works. Let $Y \subset \s$ satisfy the criteria of the lemma. As in
    \lemref{lemBadDomains}, we will prove by induction on $n$ that \[ Y,
    f^{-1}(Y), \ldots, f^{-n}(Y)=f^{L_Y+1-n}(Y)\] satisfy the lemma for
    $n=0 \ldots, L_Y$. 

    Let's first show $Y$ supports a geodesic in $H(\mu_B,\mu)$. We will be
    considering $H(f\mu_B,f\mu)$ in parallel. Note that $f\mu'$ will be a
    separating marking at $f(b)$ in $H(f\mu_B,f\mu)$. Consider
    the following pair of quadrilaterals in $\mathcal{C}(Y)$:
     %\begin{align*} &\big[ \pi_Y(\mu_B), \pi_Y(f\mu_B) \big], & \, &\big[
     %\pi_Y(\mu), \pi_Y(f\mu) \big], \\ &\big[ \pi_Y(\mu), \pi_Y(\mu)
     %\big],  & \, &\big[ \pi_Y(f\mu_B), \pi_Y(f\mu) \big]. \end{align*}
     %The segment $\big[ \pi_Y(\mu'), \pi_Y(f\mu') \big]$ divides $Q_Y$
     %into two sub-quadrilaterals: 
    \[ 
       Q_1 = 
       \big[ \mu_B,  f\mu_B  \big]_Y \cup 
       \big[ f\mu_B, f\mu'   \big]_Y \cup 
       \big[ \mu_B,  \mu'    \big]_Y \cup
       \big[ \mu',   f\mu'   \big]_Y,
    \] 
    and
    \[ 
       Q_2 = 
       \big[ \mu',  f\mu' \big]_Y \cup 
       \big[ \mu',  \mu   \big]_Y \cup 
       \big[ \mu,   f\mu  \big]_Y \cup 
       \big[ f\mu', f\mu  \big]_Y.
    \] 
    By our assumption, $d_Y(\mu',f\mu') > \Delta$. By the triangle
    inequality, at least one of three other segments of $Q_1$ is long: 
    \[
       \big[ \mu_B, \mu'   \big]_Y,   \quad 
       \big[ \mu_B, f\mu_B \big]_Y, \quad 
       \big[ f\mu_B, f\mu' \big]_Y. 
    \]
    Similarly, at least one of the following segments of $Q_2$ is long: 
    \[
       \big[ \mu',  \mu   \big]_Y, \quad 
       \big[ \mu,   f\mu  \big]_Y, \quad 
       \big[ f\mu', f\mu  \big]_Y. 
    \]
    Since $d_Z(\mu,f\mu) \le R$ for all $Z \subseteq \s$, in $Q_2$ the
    situation reduces to either 
    \begin{equation}\label{eqFO1}
      d_Y(\mu', \mu) > \frac{\Delta - R}{2} > NM_0 + NR + 5M_1,
    \end{equation}
    or 
    \begin{equation}\label{eqFO2}
      d_Y(f\mu', f\mu) > NM_0 + NR + 5M_1.  
    \end{equation}
    If \eqref{eqFO1} holds, then by the fact that $\mu'$ is a separating
    marking at $b$ (\lemref{lemSeparatingMarking}), 
    \[ d_Y(\mu_B, \mu') \le M_1. \] 
    Applying the triangle inequality yields
    \begin{align*}
      d_Y(\mu_B, \mu) 
      & \ge d_Y(\mu', \mu) - d_Y(\mu_B, \mu') \\
      & > (NM_0 + NR + 5M_1) - M_1.
    \end{align*}
    Therefore \eqref{eqFO1} implies $Y$ supports a geodesic in $H$. So we
    may assume \eqref{eqFO2} holds.
   
    In $Q_1$, the assumption on $\partial Y$ forces $\dist_\s \Big( \partial
    Y, \big[ \mu_B, f\mu_B \big] \Big) > 1$. In other words, every vertex
    in $\big[ \mu_B, f\mu_B \big]$ crosses $Y$. \thmref{thmBGI} applies and
    $d_Y(\mu_B, f\mu_B) \le M_0$. The situation is reduced to either
    \begin{equation}\label{eqFO3}
      d_Y(\mu_B, \mu') > \frac{\Delta - M_0}{2} > NM_0 + NR + 5M_1
    \end{equation}
    or
    \begin{equation}\label{eqFO4}
      d_Y(f\mu_B, f\mu') > NM_0 + NR + 5M_1.
    \end{equation}
    It is not possible for \eqref{eqFO2} and \eqref{eqFO4} to occur
    simultaneously, as that would mean both \[ d_Y(f \mu_B, f \mu') > M_1
    \quad \text{and} \quad d_Y(f\mu', f\mu) > M_1, \] violating $f\mu'$ a
    separating marking. So \eqref{eqFO3} must hold. As above, we must then
    have
    \begin{align*}
      d_Y(\mu_B, \mu) 
      & \ge d_Y(\mu_B, \mu') - d_Y(\mu',\mu)\\ 
      & > (NM_0 + NR + 5M_1) - M_1.  
    \end{align*}
    Therefore, in all cases, $Y$ must support a geodesic in $H(\mu_B,\mu)$.
    See \figref{figSubsurfaceY} for a schematic picture of $Q_1$ and $Q_2$.
    Note that the conclusion of the base case always resulted in 
    \begin{equation}\label{eqFO5}
      d_Y(\mu_B, \mu) > NM_0 + NR + 4M_1. 
    \end{equation}
    
    %\begin{figure}[htbp] 
    %  \begin{center} 
    %    \input{subsurfaceY.pstex_t} 
    %    \caption{The quadrilateral $Q_Y$ in $\mathcal{C}(Y)$}
    %    \label{figSubsurfaceY} 
    %  \end{center} 
    %\end{figure}

    \begin{figure}
      
      \begin{center}

        \begin{tikzpicture}[yscale=1.6]

          \draw[thick, dotted, rounded corners=10pt]
          (1.6,0.6) -- (11.4,0.4);
          
          \draw[thick, rounded corners=24pt]
          (0,0) -- (0.6,0.4) -- (12.4,0.4) -- (13,0.3);
          \draw[thick, rounded corners=24pt]
          (0,1) -- (0.6,0.6) -- (12.4,0.6) -- (13,0.7);

          \draw[thick, rounded corners=11pt]
          (0,0) -- (0.8,0.5) -- (0,1);
          \draw[thick, rounded corners=11pt]
          (13,0.3) -- (12.5,0.5) -- (13,0.7);

          \filldraw[fill=green, draw=black, thick, yscale=0.625] 
          (1.6,0.96)  circle (2.5pt);
          \filldraw[fill=green, draw=black, thick, yscale=0.625] 
          (11.4,0.64) circle (2.5pt);

          \draw (0,0)      node[anchor=north]      {$\pi_Y(\mu_B)$};
          \draw (13,0.3)     node[anchor=north]      {$\pi_Y(\mu)$};
          \draw (11.4,0.3) node[anchor=north east] {$\pi_Y(\mu')$};
          \draw (0,1)      node[anchor=south]      {$\pi_Y(f\mu_B)$};
          \draw (13,0.7)     node[anchor=south]      {$\pi_Y(f\mu)$};
          \draw (1.6,0.7)  node[anchor=south west] {$\pi_Y(f\mu')$};

        \end{tikzpicture}

      \end{center}
       
      \caption{The quadrilaterals $Q_1$ and $Q_2$ in $\mathcal{C}(Y)$}
      
      \label{figSubsurfaceY} 
    
    \end{figure}
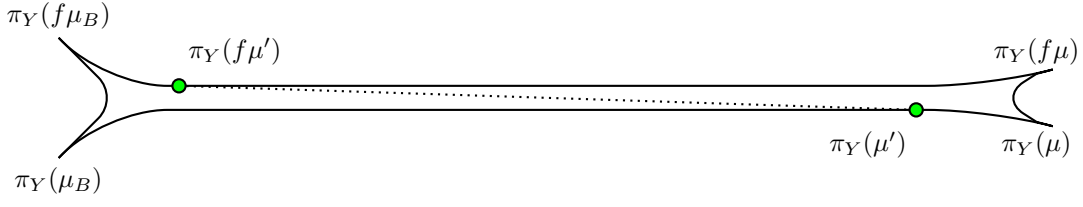

    By induction, \[ Y, f^{-1}(Y), \ldots, f^{-n+1}(Y) \] are all domains
    for geodesics in $H(\mu_B,\mu)$ and are all mutually disjoint. Let's
    now prove $f^{-n}(Y)$ supports a geodesic in $H(\mu_B, \mu)$. Since
    \[
       \dist_\s \Big( \big[ v_B,f(v_B) \big], \big[ b, f(b) \big] \Big) \ge
       4 \quad 
       \text{and} 
       \quad \dist_\s \Big( \partial Y, \big[ b,f(b) \big] \Big) \le 1, 
    \] 
    the disjointness condition will imply
    \[ 
      \dist_\s \Big( \big[ v_B, f(v_B) \big], \partial f^{-i}(Y) \Big) \ge
      2, 
    \] 
    for all $i = 0, \ldots, n-1$. By \thmref{thmBGI},  
    \begin{equation}\label{eqFO6}
      d_{f^{-i}(Y)}(\mu_B, f\mu_B) \le M_0.
    \end{equation}
    Coupling this fact with \eqref{eqFO5} (in the last step below), we have 
    \begin{align*}
      d_{f^{-n}(Y)} &(\mu_B, \mu) \\ 
      & = d_{f^{-n+1}(Y)}(f\mu_B, f\mu) \\
      & \ge d_{f^{-n+1}(Y)}(\mu_B, \mu) - d_{f^{-n+1}(Y)}(\mu_B, f\mu_B) -
      d_{f^{-n+1}(Y)}(\mu, f\mu) \\ 
      & \ge d_{f^{-n+1}(Y)}(\mu_B, \mu) - M_0 - R \\
      & \hspace{1cm} \vdots \\
      & \ge d_Y(\mu_B, \mu) - nM_0 - nR \\
      & > (N-n)M_0 + (N-n)R + 4M_1. 
    \end{align*}
    Since $N > L_Y \ge n$, the above in particular implies 
    \[
       d_{f^{-n}(Y)}(\mu_B,\mu) > M_1 \ge M_2. 
    \] 
    So $f^{-n}(Y)$ supports a geodesic in $H(\mu_B,\mu)$. 
   
    We now want to show $Y, \ldots, f^{-n+1}(Y), f^{-n}(Y)$ are all
    pairwise disjoint. Using the action of $f$ and the assumption that $Y,
    \ldots, f^{-n+1}(Y)$ are pairwise disjoint, we see that $f^{-n}(Y)$ is
    disjoint with each $Y, \ldots, f^{-n+1}(Y)$ if and only if $f^{-n}(Y)$
    and $Y$ are disjoint. If $Y$ and $f^{-n}(Y)$ are not disjoint, then
    they are time-ordered in $H(\mu_B,\mu)$. The two different cases of
    time-ordering of $Y$ and $f^{-n}(Y)$ will both lead to a contradiction.
    The argument is very similar to the one given in
    \lemref{lemFiniteOrder}. We will quickly give the argument in the case
    that $Y <_t f^{-n}(Y)$ and omit the case the argument in the second
    case.

    Suppose $Y <_t f^{-n}(Y)$. Let $\nu$ be a hierarchy marking $H(\mu_B,
    \mu)$ such that 
    \begin{equation}\label{eqHierMarking2} 
       d_Y(\nu, \mu) \le M_1 \quad \text{and} \quad
       d_{f^{-n}(Y)}(\mu_B,\nu) \le M_1, 
    \end{equation}
    as in \lemref{lemTimeOrder2}. Then 
    \begin{align*}
      d_Y(\mu_B,f^{n}\nu) 
      & \le d_Y(\mu_B, f^n \mu_B) + d_Y(f^n \mu_B,f^n\nu) \\ 
      & \le \left( \sum_{j=0}^{n-1} d_Y(f^j\mu_B, f^{j+1} \mu_B) \right) +
      d_{f^{-n}(Y)}(\mu_B, \nu) \\ 
      & \le \left( \sum_{j=0}^{n-1}d_{f^{-j}(Y)}(\mu_B,f\mu_B) \right) + M_1 \\
      & \le nM_0 + M_1.  
    \end{align*}
    Using \eqref{eqFO5} and the triangle inequality, we have 
    \[
       d_Y(f^n\nu,\mu) \ge d_Y(\mu_B,\mu) - d_Y(\mu_B, f^n\nu) > 2M_1. 
    \]
    Therefore, by \lemref{lemTimeOrder2}, 
    \[ d_{f^{-n}(Y)}(\mu_B, f^{n}\nu) \le 2M_1.  \]
    By considering powers of $f^n\nu$ inductively, we have
    \begin{align*} 
      d_Y(\mu_B,f^{(i+1)n}\nu) 
      & \le d_Y(\mu_B, f^n\mu_B) + d_Y(f^n\mu_B, f^{(i+1)n}\nu) \\ 
      & \le nM_0 + d_{f^{-n}(Y)}(\mu_B,f^{in}\nu) \\ 
      & \le nM_0 + 2M_1 
    \end{align*}
    This is true for every $i \ge 0$. Since $N > L_Y \ge n$, using
    \eqref{eqHierMarking2}, we have
    \[ 
      d_Y(\mu_B, \mu) \le d_Y(\mu_B,\nu) + d_Y(\nu, \mu) \le nM_0 + 3M_1,
    \] 
    contradicting \eqref{eqFO5}. The case of $f^{-n}(Y) <_t Y$ will lead to
    a similar contradiction. This concludes the proof of the lemma.
  \end{proof}

\section{L.B.C.~property for finite order mapping classes}

  \label{secFinite}
  
  The heart of this section is to prove \thmref{introthmPeriodic} in the
  introduction, which is restated below. Let $R_1$ be the fixed constant of
  \lemref{lemFixedPoints} and let $\mu_B$ be the fixed base marking. 

  \begin{theorem} \label{thmFiniteOrder}
    
    There exists a constant $R \ge R_1$, depending only on $\mu_B$, such
    that any finite order $f \in \mcg$ has a marking $\mu \in \wfix_R(f)$
    with 
    \[
      d_{\mg}(\mu_B, \mu) \prec d_{\mg}(\mu_B, f\mu_B). 
    \] 

  \end{theorem}

  Assuming \thmref{thmFiniteOrder}, we can derive L.B.C.~property for
  finite order mapping classes by a standard argument, following \cite{BH99}.
  We first state and prove the following corollary of
  \thmref{thmFiniteOrder}, which reduces L.B.C.~property for finite order
  mapping classes to a finite problem.
  
  \begin{corollary}\label{corFiniteOrder}

    There exists a finite set $\Gamma \subset \mcg$ such that, for every finite
    order $f \in \mcg$, there exists $\omega \in \mcg$ such that $\omega^{-1} f
    \omega \in \Gamma$ and $ |\omega| \prec |f|$.

  \end{corollary}
  
  \begin{proof}

  By enlarging $R$ if necessary, we may rephrase \thmref{thmFiniteOrder} in
  terms of fixed points: there exists $R$ depending only on $\mu_B$ such that
  any finite order mapping class $f$ has a marking $\mu \in \fix_R(f)$ with
  \begin{equation}\label{eqFix} d_{\mg}(\mu_B, \mu) \prec d_{\mg}(\mu_B, f\mu_B).
  \end{equation}

  We construct the set $\Gamma$ as follows. Let $D$ be the diameter of
  $\mg$ modulo the action of $\mcg$. ($D$ is finite since the action of
  \mcg on \mg is cofinite). Set \[ \Gamma = \{\,\, g \in \mcg
      \,\,\,:\,\,\, d_{\mg}(\mu_B, g\mu_B) \le 2D + R, \, g \textrm{ finite order}
  \,\, \}.\] The action of $\mcg$ on $\mg$ is proper, thus $\Gamma$ is a
  finite set. We show $\Gamma$ satisfies the other properties as well.

  Let $f \in \mcg$ be of finite order. Let $\mu$ be a $R$-fixed point of $f$
  closest to $\mu_B$. Since the action of $\mcg$ on $\mg$ is cofinite, there
  exists $\omega \in \mcg$ such that $d_{\mg}(\omega\mu_B,\mu) \le D$. Then
  $\omega^{-1}f\omega \in \Gamma$, since
  \begin{align*}
    d_{\mg}(\mu_B, \omega^{-1}f\omega\mu_B) 
    & = d_{\mg}(\omega \mu_B, f \omega \mu_B) \\
    & \le d_{\mg}(\omega \mu_B, \mu) + d_{\mg}(\mu,f\mu) + d_{\mg}(f\mu, f\omega \mu_B) \\ 
    & \le d_{\mg}(\omega \mu_B, \mu) + d_{\mg}(\mu,f\mu) + d_{\mg}(\mu, \omega \mu_B) \\
    & \le 2D + R. 
  \end{align*}
  Moreover, by \eqref{eqFix} we have 
  \begin{align*}
    |\omega| 
    & \prec d_{\mg}(\mu_B,\omega\mu_B) \\ 
    & \le d_{\mg}(\mu_B,\mu) + d_{\mg}(\omega\mu_B,\mu) \\ 
    & \prec d_{\mg}(\mu_B, f\mu_B) + D \\
    & \prec |f|. \qedhere
  \end{align*} 
  \end{proof}
  
  \begin{corollary}[L.B.C.~property for finite order mapping
    classes]\label{corFiniteOrder}

   If $f, g \in \mcg$ are conjugate finite order mapping classes, then there is
   a conjugating element $\omega \in \mcg$ with \[ |\omega| \prec |f| + |g|. \]

  \end{corollary}
  
  \begin{proof}
     
     Let $\Gamma \subset \mcg$ be the finite set of
     \thmref{thmFiniteOrder}. The content of \thmref{thmFiniteOrder} is
     that $\Gamma$ contains at least one and at most finitely many
     representatives for each conjugacy class of a finite order mapping
     class. Furthermore, each finite order $f$ can be conjugated into
     $\Gamma$ by a conjugating element whose word length is proportional to
     $|f|$. The result follows after picking a conjugating element for each
     pair of elements in $\Gamma$ of the same conjugacy class.
  \end{proof}

  The proof of \thmref{thmFiniteOrder} will occupy the rest of the section.
  The main observation is that if $\mu_1 \in \wfix_{R_1}(f)$ does not have
  any $R_1$--bad domains, then $\mu_1$ satisfies the statement of
  \thmref{thmFiniteOrder} (see \propref{propEmpty}).  If $\mu_1$ does have
  a $R_1$-bad domain $Y$, then we can construct a marking $\mu_2 \in
  \wfix_{R_2}(f)$, where $R_2$ depends only on $R_1$, such that $Y \notin
  \Omega(\mu_2,R_2)$. Ideally, we would like $\Omega(\mu_2,R_2)$ to be
  strictly smaller than $\Omega(\mu_1,R_1)$, but the situation is a bit
  more complicated. In trying to improve $\mu_1$ in $Y$, we may have
  created new bad domains but we will have control over what they are in
  relation to $Y$. We call this the base step of the proof. Although the
  set of bad domains are not necessarily decreasing, applying the base step
  in the right way will guarantee  a decrease in the complexity of
  $\Omega(\mu_2, R_2)$ from that of $\Omega(\mu_1,R_1)$. By iterating this
  process, we produce a sequence of symmetric points for $f$ such that the
  complexities of the sets of bad domains are monotonically decreasing.
  This process must stop to produce an $R$--symmetric point $\mu$ for $f$
  with no bad domains. Since the maximal complexity of any set of bad
  domains is the complexity of \s, the process of achieving $\mu$
  terminates after a bounded number of steps. This serves to ensure the
  constant $R$ will depend only on $R_1$ (and $\mu_B$).
  
  The rest of the section is organized as follows. We first prove in
  \propref{propEmpty} that no bad domain indeed implies
  \thmref{thmFiniteOrder}. Then the base step is dealt with in
  \secref{secBaseStep}. There are two propositions, \propref{propBaseStep1}
  and \propref{propBaseStep2}, associated to the base step, depending on
  whether the bad domain is the main surface or a proper subsurface. This
  is where our work in \secref{secTwoLemmas} will come in. In
  \secref{secIteration}, we explain how to use the base step to reduce
  complexity of the set of bad domains. The precise statement is
  \propref{propReduction}. The section will conclude with
  \corref{corTermination} which makes precise how the process terminates
  after a bounded number of steps.

  \subsection{No bad domains} \label{secEmpty} 
   
  \begin{proposition}[No bad domains]\label{propEmpty} 
   
   Let $\mu \in \wfix_R(f)$ where $R \ge R_1$. If $R$ is a constant
   depending only on $R_1$ such that $\Omega(\mu,R) = \emptyset$, then \[
   d_{\mg}(\mu_B, \mu) \prec d_{\mg}(\mu_B, f\mu_B). \] In other words, $\mu$ and
   $R$ satisfy \thmref{thmFiniteOrder}.

  \end{proposition}

  \begin{proof} 
    
    The assumption $\Omega(\mu, R) = \emptyset$ means that for every $X
    \subseteq \s$, there exists $i_X$ such that 
    \begin{equation*}
    d_X(\mu_B, \mu) \le 2N d_{f^{i_X}(X)}(\mu_B, f\mu_B) + NR + \Theta.
    \end{equation*} 
    Let $L_0$ be the constant of the distance formula, \thmref{thmDF}. Let 
    \[ 
      \Phi
      = \{\,\, X \subseteq \s \,\,\, : \,\,\, d_X(\mu_B,\mu) 
      \ge 2NL_0 + NR + \Theta \,\,\}, 
    \] 
    and 
    \[ 
       \Psi = \{ \,\, Y \subseteq \s \,\,\, : \,\,\, d_Y(\mu_B,f\mu_B) \ge
       L_0 \,\,\}. 
    \] 
    Then there is a map $\Phi \to \Psi$ sending $X \mapsto f^{i_X}(X)$.
    This map has multiplicity at most the order of $f$, which is bounded by
    $N$. Therefore, 
    \begin{align*} 
      d_{\mg}(\mu_B,\mu) 
      & \asymp \sum_{X \in \Phi} d_X(\mu_B,\mu) \\
      & \le \sum_{X \in \Phi} 2N d_{f^{i_X}(X)}(\mu_B,f\mu_B) +
      NR + \Omega\\ 
      & \prec \sum_{X \in \Phi} d_{f^{i_X}(X)}(\mu_B,f\mu_B) \\  
      & \le N \sum_{Y \in \Psi} d_Y(\mu_B,f\mu_B) \\
      & \asymp d_{\mg}(\mu_B,f\mu_B). \qedhere
    \end{align*}
  \end{proof}

  \subsection{Base step}\label{secBaseStep} 
  
  We are now ready to state and prove the base step of the proof for
  \thmref{thmFiniteOrder}. There are two cases to consider, which are
  \propref{propBaseStep1} and \propref{propBaseStep2}. The proof of
  \propref{propBaseStep1} will be essential for \propref{propBaseStep2}.
  
  Let $\mu_B$ be the base marking in \mg, and recall the definition of $\xi
  \big( \Omega(\mu,R) \big)$ as in \defref{defComplexity}.

  \begin{proposition}[Base Step 1]\label{propBaseStep1}
    
    Given $R_I \ge \max\{ 2\delta+4, R_1 \}$ there exists a constant $R_O$
    depending only on $R_I$ with the following property. Given $\mu_I \in
    \wfix_{R_I}(f)$, if $\s \in \Omega(\mu_I, R_I)$, then there exists
    $\mu_O \in \mg$ satisfying the following properties: 
    \begin{enumerate}

    \item[(P1)] $\mu_O \in \wfix_{R_O}(f)$.

    \item[(P2)] $\Omega(\mu_O, R_O) \subsetneqq \Omega(\mu_I, R_I)$. In
      addition, $\s \notin \Omega(\mu_O, R_O)$, and thus $\xi(\mu_O, R_O)
      \lneqq \xi(\mu_I, R_I)$.

  \end{enumerate}

  \end{proposition}

  \begin{proof} 
    
    We have four hierarchies: 
    \[
       H(\mu_B,f\mu_B), \quad H(\mu_I,f\mu_I), \quad H(\mu_B,\mu_I), \quad
       H(f\mu_B,f\mu_I). 
    \] 
    Consider the four main geodesics corresponding to the
    four hierarchies, forming a quadrilateral $Q$ in $\cc$: \[ \big[
        v_B,f(v_B)
    \big], \quad \big[ v_I,f(v_I) \big], \quad \big[ v_B,v_I \big], \quad \big[
    f(v_B),f(v_I) \big],\] where $v_B$ and $v_I$ are base curves in $\mu_B$ and
    $\mu_I$, respectively. Our assumption is that $\s \in \Omega(\mu_I,R_I)$, so
    \[ d_\s(v_B,v_I) > d_\s \big( v_B,f(v_B) \big) + R_I. \] 

    Since $f$ acts on $\cc$ as an isometry, $d_\s(v_B,v_I) = d_\s \big(
    f(v_B),f(v_I) \big)$, so $Q$ is $2\delta$--thin: every edge of $Q$ is
    contained in a $2\delta$--neighborhood of the other edges. The
    geodesics $[v_I, v_B]$ and $[f(v_I), f(v_B)]$ $2\delta$--fellow travel
    for awhile until $[v_I, v_B]$ begins fellow traveling $[v_B, f(v_B)]$.
    Choose the vertex $b$ on $[v_I, v_B]$ at the junction where this change
    takes place. After possibly moving $b$ toward $v_I$, by at most
    $2\delta+4$ positions, we may assume the following properties for $b$
    (see \figref{figSurfaceS}): 
    \begin{itemize}
    
      \item $d_\s \big( b,f(b) \big) \le 2 \delta$, and 
    
      \item $4 \le \dist_\s \Big( \big[ b,f(b) \big], \big[ v_B,f(v_B) \big]
      \Big) \le 6\delta + 4$. 

    \end{itemize}
    By the triangle inequality, we have 
    \begin{equation}\label{eqBS1a}
      d_\s(b,v_B) \le d_\s \big( v_B,f(v_B) \big) + 8\delta + 4.
    \end{equation}

    \begin{figure}
    \begin{center}
      \begin{tikzpicture}

        \draw[thick, rounded corners=24pt]
        (0,1.5) -- (2,0.15) -- (8,0.15);
        \draw[thick, rounded corners=24pt]
        (9,0.15) -- (10,0.15) -- (13,0.3);
        \draw[thick]
        (8,0.15) -- (9,0.15);
        
        \draw[thick, rounded corners=24pt]
        (0,-1.5) -- (2,-0.15) -- (8,-0.15);
        \draw[thick, rounded corners=24pt]
        (9,-0.15) -- (10,-0.15) -- (13,-0.3);
        \draw[thick]
        (8,-0.15) -- (9,-0.15);

        \draw[thick, rounded corners=10pt]
        (0,1.5) -- (2,0) -- (0,-1.5);
        \draw[thick, rounded corners=10pt]
        (13,0.3) -- (12.5,0) -- (13,-0.3);

        \draw[thick, rounded corners=10pt]
        (2.3,0.23) -- (1.9,0) -- (2.3,-0.23);

        \filldraw[fill=green, draw=black, thick] (2.3,0.23)  circle (2.5pt);
        \filldraw[fill=green, draw=black, thick] (2.3,-0.23) circle (2.5pt);

        \draw (0,1.5)     node[anchor=south]      {$f(v_B)$};
        \draw (0,-1.5)    node[anchor=north]      {$v_B$};
        \draw (13,0.3)    node[anchor=south]      {$f(v_I)$};
        \draw (13,-0.3)   node[anchor=north]      {$v_I$};
        \draw (2.3,0.23)  node[anchor=south west] {$f(b)$};
        \draw (2.3,-0.23) node[anchor=north west] {$b$};

      \end{tikzpicture}
    \end{center}
    \caption{The quadrilateral $Q$ in $\cc$}
    \label{figSurfaceS}
    \end{figure}
    
    Now choose the separating marking $\mu$ in $H(\mu_B,\mu_I)$ at $b$.
    The proof divides into two cases. To describe these cases, consider any
    domain $Y \subset \s$ for which 
    \begin{equation}\label{eqBS1}
      \dist_\s \big( \partial Y, [ b,f(b)] \big) \le 1,
    \end{equation}
    and let 
    \[ 
       \Delta = (2N+1) M_0 + (2N+1)R_I + 10M_1. 
    \] 
    \noindent \textbf{Case I.} Suppose $d_Y(\mu, f\mu) \le \Delta$ for all
    $Y \subset \s$ satisfying \eqref{eqBS1}. In this case, set $\mu_O=\mu$
    and $R_O = \Delta$. We show $\mu_O$ and $R_O$ satisfy
    \propref{propBaseStep1}. First note that $\s \notin \Omega(\mu_O, R_O)$
    by \eqref{eqBS1a}. Since $\mu_O$ is a hierarchal marking in $H(\mu_B,
    \mu_I)$, we also have, for all $Z \subset \s$, \[ d_Z(\mu_B,\mu_O) \le
      d_Z(\mu_B,\mu_I) + M_3. 
    \] 
    Since $M_3 < M_1 < R_O$, this verifies (P2).
    
    To see property (P1), we consider $d_Z(\mu_O, f\mu_O)$ for three
    possibilities of $Z$.
    \begin{enumerate}
    
      \item[(a1)] If $d_\s \big( \partial Z, [ b,f(b)] \big) \le 1$, then \[
        d_Z(\mu_O, f\mu_O) \le \Delta = R_O. \]
       
      \item[(a2)] If $d_\s \big( \partial Z, [ b,f(b) ] \big) > 1$, then every
        vertex of $\big[ b, f(b) \big]$ cuts $Z$. By \thmref{thmBGI}, 
        \[
           d_Z \big(b,f(b)\big) \le M_0 \quad \Longrightarrow \quad
           d_Z(\mu_O,
           f\mu_O) \le M_0 + 4. 
        \]
   
      \item[(a3)] If $Z = \s$, then by construction, $d_\s \big( b,f(b)
      \big) \le 2\delta$, thus \[ d_\s(\mu_O, f\mu_O) \le 2\delta + 4.\]
    
    \end{enumerate}
    This ends the proof of the proposition in Case I.

    \noindent \textbf{Case II.} Suppose there exists $Y$ with $d_Y(\mu,
    f\mu) > \Delta$ for some $Y \subset \s$ satisfying \eqref{eqBS1}. In
    this case, \lemref{lemFiniteOrder} implies that $Y$ and all its orbits
    under $f$ are pairwise disjoint. Consider the multicurve \[ c =
      \partial Y \cup \partial f(Y) \cup \cdots \cup \partial f^{L_Y}(Y).
    \] Let $\mu_O$ be a marking extension of $c$ relative to $\mu_I$, as
    in \defref{defMarkingExtension}. In particular, $c \subseteq
    \base(\mu_O)$. Set \[ R_O = \max \big\{ R_I + 2M_3, 10\delta + 13 \big\}.
    \] Consider the following properties for $\mu_O$ and $R_O$.
    \begin{enumerate}
    
      \item[(b1)] If $Z = \s$, then, since both $\mu_O$ and $f\mu_O$
      contain $c$ as base curves, 
      \begin{equation} \label{eqbone}
         d_\s(\mu_O, f\mu_O) \le d_\s(\mu_O, c) + d_\s(c, f\mu_O) \le 4.
      \end{equation}
      Also, since \[ \dist_\s \Big(
      \partial Y, \big[ b,f(b) \big] \Big) \le 1,\] we have
        \begin{align}\label{eqC3}
          d_\s(\mu_B, \mu_O) 
          & \le d_\s(v_B, \partial Y) + 4\nonumber \\
          & \le d_\s(v_B, b) + d_\s \big( b, f(b) \big) + 5 \nonumber \\ 
          & \le \Big( d_\s \big(v_B,f(v_B) \big)+8\delta+4 \Big) + 2\delta
          + 5 \nonumber \\ 
          & \le d_\s \big( v_B,f(v_B) \big) + 10\delta + 9 \nonumber \\  
          & \le d_\s \big( \mu_B, f(\mu_B) \big) + 10 \delta + 13.
        \end{align}

      \item[(b2)] If $Z \ne \s$, but some curve $\alpha$ in $c$ crosses $Z$, then 
      \begin{equation} \label{eqbtwo}
         d_Z(\mu_O, f\mu_O) \le d_Z( \mu_O, \alpha) + d_Z (\alpha, f\mu_O) \le 8. 
       \end{equation}
      Furthermore, according to \lemref{lemFiniteOrder}, $f^i(Y)$ all support a
      geodesic in $H(\mu_B, \mu_I)$, so there exists a slice of
      $H(\mu_B,\mu_I)$ containing $\alpha$. Therefore,
      \begin{align} \label{eqbtwob}
          d_Z(\mu_B, \mu_O) 
          & \le d_Z(\mu_B, \alpha) + 4 \nonumber \\
          & \le d_Z(\mu_B, \mu_I) + M_3 + 4.
        \end{align}

      \item[(b3)] If $Z \ne \s$ is such that $Z$ is disjoint from $c$
      or $Z$ is curve in $c$, then we are in the situation of
      \lemref{lemMarkingExtension}. Note that $f\mu_O$ is a marking
      extension of $c$ relative to $f\mu_I$. 
      \begin{align} \label{eqbthree}
         d_Z(\mu_O, f\mu_O) 
         & \le d_Z(\mu_O,\mu_I) + d_Z(\mu_I,f\mu_I) + d_Z(f\mu_I,f\mu_O)
         \nonumber\\ 
         & \le d_Z(\mu_I, f\mu_I) + 2M_3 \nonumber \\
         & \le R_I + 2M_3 \nonumber \\
         & \le R_O,
      \end{align}
       and 
       \begin{equation} \label{eqbthreeb} 
          d_Z(\mu_B, \mu_O) \le d_Z(\mu_B,\mu_I) + d_Z(\mu_I,\mu_O) \le
          d_Z(\mu_B, \mu_I) + M_3. 
        \end{equation} 
    
    \end{enumerate}
   
    Using the analyses of (b1) to (b3), we verify properties (P1) and (P2)
    for $\mu_O$ and $R_O$. From \eqref{eqbone}, \eqref{eqbtwo},
    \eqref{eqbthree}, we have that, for any $Z \subseteq \s$, $d_Z(\mu_O,
    f\mu_O) \le R_O$. Thus $\mu_O \in \wfix_{R_O} (f)$ and (P1) is
    verified. To see (P2), first note that, by \eqref{eqC3}, we have $\s
    \notin \Omega(\mu_O, R_O)$. Now, if $Z \in \Omega(\mu_O, R_O$) where $Z
    \subset \s$, then $Z$ must be of case (b2) or (b3). In either case,
    using \eqref{eqbtwob} or \eqref{eqbthreeb}, we obtain 
    \begin{align*}
      d_Z(\mu_B, \mu_I) 
      & \ge d_Z(\mu_B,\mu_O) - (M_3+4)\\
      & > 2N \Big( \max_i \big\{ d_{f^i(Z)}(\mu_B, f\mu_B) \big\} \Big) +
      NR_O + \Theta - (M_3+4) \\ 
      & > 2N \Big( \max_i \big\{ d_{f^i(Z)}(\mu_B, f\mu_B) \big\} \Big) + N(R_I +
      2M_3) + \Theta -(M_3+4)\\ 
      & > 2N \Big( \max_i \big\{ d_{f^i(Z)}(\mu_B, f\mu_B) \big\} \Big) + NR_I +
      \Theta.
    \end{align*}
    Therefore, $\Omega(\mu_O, R_O) \subset \Omega(\mu_I,R_I)$, establishing
    (P2). This finishes the proof the proposition in Case II.
  \end{proof}

  %\begin{remark}\label{remBaseStep1}

  %  Let $\mu_O$ and $R_O$ be the pair satisfying \propref{propBaseStep1}
  %  coming from either case I or II in the proof. We emphasize that that in
  %  establishing the inclusion $\Omega(\mu_O, R_O) \subset \Omega(\mu_I,
  %  R_I)$, we essentially proved that for all $Z \subset \s$, \[
  %  d_Z(\mu_B, \mu_O) \le d_Z(\mu_B, \mu_I) + M_3+4. \] The inclusion was
  %  strict since $\s \notin \Omega(\mu_O, R_O)$ by construction.

  %\end{remark}

  %\begin{proposition}[Base Step 2]\label{propBaseStep2}

  %  Let $R_I$ and $\mu_I$ be as in \propref{propBaseStep1}. If $\s \notin
  %  \Omega(\mu_I, R_I)$ but $\Omega(\mu_I,R_I)$ contains a proper domain $X
  %  \subset \s$, then there exists a constant $R_O$, depending only on
  %  $R_I$, and $\mu_O \in \mg$, satisfying the following properties:
  %  \begin{enumerate}

  %  \item[(Q1)] $\mu_O \in \wfix_{R_O}(f)$.

  %  \item[(Q2)] Let $ c = \partial X \cup \partial f(X) \cup \cdots \partial
  %    f^{L_X}(X).$  Then $c \subseteq \base(\mu_O)$. 

  %  \item[(Q3)] For all $i = 0, \ldots L_X$, $f^i(X) \notin \Omega(\mu_O, R_O)$.
  % 
  %  \item[(Q4)] Suppose $Z \in \Omega(\mu_I, R_I)$ has the property that $Z$
  %    interlocks $f^i(X)$, for some $0 \le i \le L_X$. If $X <_t f^{-i}(Z)$
  %    in $H(\mu_B, \mu_I)$, then $Z \notin \Omega(\mu_O, R_O)$.

  %  \item[(Q5)] If $Z \in \Omega(\mu_O, R_O)$ but $Z \notin \Omega(\mu_I,
  %  R_I)$, then $Z$ must be a subsurface of $f^i(X)$, for some $0 \le i \le
  %  L_X$. In particular, $\xi(Z) < \xi(X)$.

  %\end{enumerate}

  %\end{proposition}

  Before we state the next proposition we will need some notations and
  definitions. Given proper domains $X, Y \subset \s$, let 
  \[ 
     \mathcal{U} = X \cup \cdots \cup f^{L_X}(X) \quad \text{ and } \quad
     \mathcal{V} = Y \cup \cdots \cup f^{L_Y}(Y). 
  \] 
  We will say $Y$ is \emph{supported} on $\s \setminus \mathcal{U}$ if $Y$
  lies in some component of $\s \setminus \mathcal{U}$. In the case that
  $X$ is not a curve, $Y$ can be a boundary curve of $f^i(X)$, for some $0
  \le i \le L_X$. Note the symmetry in the definition: if $Y$ is supported
  on $\s \setminus \mathcal{U}$ then $X$ is supported on $\s \setminus
  \mathcal{V}$. Furthermore, if $Y$ is supported on $\s \setminus
  \mathcal{U}$, then so is $f^j(Y)$ for all $j=0, \ldots L_Y$. Thus it
  makes sense to say that $\mathcal{U}$ and $\mathcal{V}$ are disjoint.
  Similarly, given $X_1 \ldots X_n \subset \s$ and let \[ \mathcal{U}_i =
  X_i \cup \cdots \cup f^{L_{X_i}}(X_i),\] we will say $\mathcal{U}_1,
  \ldots \mathcal{U}_n$ are pairwise disjoint if, for all $1 \le i, j \le
  n$, $i \ne j$, $X_i$ is supported on $\s \setminus \mathcal{U}_j$. 

  \begin{proposition}[Base Step 2]\label{propBaseStep2}

    Given $R_I \ge \max \{2\delta +4, R_1\}$ there exists a constant $R_O$
    depending only on $R_i$ with the following property. Given $\mu_I \in
    \wfix_{R_I}(f)$ and suppose $\s \notin \Omega(\mu_I, R_I)$. If
    $\Omega(\mu_I,R_I)$ contains proper domains $X_1, \ldots X_n \subset
    \s$ such that $\mathcal{U}_1, \ldots, \mathcal{U}_n$ are pairwise
    disjoint, where $\mathcal{U}_i = X_i \cup \cdots \cup
    f^{L_{X_i}}(X_i)$, then there exists $\mu_O \in \mg$, satisfying the
    following properties:
    \begin{enumerate}

    \item[(Q1)] $\mu_O \in \wfix_{R_O}(f)$.

    \item[(Q2)] For $j=1,\ldots n$, let $c_j = \partial X_i \cup \partial
      f(X_i) \cup \cdots \partial f^{L_{X_i}}(X_i)$. Then $c = \bigcup_j
      c_j \subseteq \base(\mu_O)$. 

    \item[(Q3)] For all $j=1, \ldots n$ and all $i = 0, \ldots L_{X_j}$,
      $f^i(X_j) \notin \Omega(\mu_O, R_O)$.
   
    \item[(Q4)] Suppose $Z \in \Omega(\mu_I, R_I)$ has the property that $Z$
      interlocks $f^i(X_j)$, for some $0 \le j \le n$ and some $0 \le i \le
      L_{X_j}$. If $X_j <_t f^{-i}(Z)$ in $H(\mu_B, \mu_I)$, then $Z \notin
      \Omega(\mu_O, R_O)$.

    \item[(Q5)] If $Z \in \Omega(\mu_O, R_O)$ but $Z \notin \Omega(\mu_I,
    R_I)$, then $Z$ must be a subsurface of $f^i(X_j)$, for some $0 \le j
    \le n$ and $0 \le i \le L_{X_j}$. In particular, $\xi(Z) < \xi(X_j)$.

  \end{enumerate}

  \end{proposition}
  
  \begin{remark}\label{remBaseStep2}

    We briefly explain the statements in (Q1) to (Q5). 
    
    First note that, by \lemref{lemBadDomains}, for each $j=1, \ldots n$,
    \[ X_j, f(X_j), \ldots, f^{L_{X_j}}(X_j) \] are all pairwise disjoint.
    Since the $\mathcal{U}_i$'s are assumed to be pairwise disjoint, the
    set $c=\bigcup_j c_j$ is a multicurve on \s, so property (Q2) makes
    sense. (Recall that if $X_j$ is a curve, then $\partial X_j = X_j$.)

    Secondly, the assumption in (Q4) also makes sense. If $Z \in
    \Omega(\mu_I, R_I)$ interlocks $f^i(X_j)$, then $X_j$ and $f^{-i}(Z)$
    interlock by the action of $f$. Since they both support geodesics in
    $H(\mu_B, \mu_I)$ (\lemref{lemBadDomains}), they must be time-ordered.
    The proposition analyzes the case when $X_j <_t f^{-i}(Z)$. 

    The point of property (Q3) is that, if $X_j$ is a bad domain for
    $\mu_I$, then we can improve $\mu_I$ in $X_j, f(X_j), \ldots
    f^{L_{X_j}}(X_j)$ simultaneously. This process also eliminates all bad
    domains of type specified by (Q4). However, during this process, a new
    bad domain $Z$ which was not a  bad domain for $\mu_I$ may have been
    created. Property (Q5) puts restrictions on such $Z$: namely, $Z$ must
    be a subsurface of some $f^i(X_j)$, which has strictly smaller
    complexity than that of $X_j$. If $X_1, X_2, \ldots X_n$ are all
    curves, then in particular (Q5) implies such $Z$ cannot exist and thus
    $\Omega(\mu_O, R_O) \subset \Omega(\mu_I,R_I)$.
    
  \end{remark}

  \begin{proof}[Proof of \propref{propBaseStep2}] 

    We will first assume that $n=1$ and set $X = X_1$. We will construct a
    marking $\mu_O$ containing \[ c = \partial X \cup \cdots \cup \partial
    f^{L_X}(X). \] as base curves, guaranteeing (Q2). The situation may
    seem similar to case II of \propref{propBaseStep1}, but to ensure (Q3),
    it will not be enough to construct $\mu_O$ by inducing $\mu_I$ on each
    $f^i(X)$. We will in fact need the full work of \propref{propBaseStep1}
    to construct a marking on $X$. The action of $f$ will then extend this
    marking to each $f^i(X)$. We will consider two cases, when $X$ is a
    curve or when $X$ a non-annular subsurface. The two cases are pretty
    much the same, but for clarity, we treat them separately. After we
    explain how to construct $\mu_O$ and $R_O$ in each case, we will then
    check that they satisfy the proposition.
    
    First suppose $X$ is a curve. On each component domain of $(S,c)$, we
    put the induced marking coming from $\mu_I$. To complete this into a
    marking, we need to pick a transversal to each $f^i(X)$. Much like as
    in the proof of \propref{propBaseStep1}, we have a quadrilateral in
    $\mathcal{C}(X)$ formed by projecting the main geodesics $\mu_B$,
    $f\mu_B$, $\mu_I$, and $f\mu_I$ to $\mathcal{C}(X)$. Since the pair of
    geodesics $\big[ \mu_B, \mu_I \big]_X$ and $\big[ f\mu_B, f\mu_I
    \big]_X$ $2\delta$--fellow travel in $\mathcal{C}(X)$, we can find an
    element $b \in \mathcal{C}(X)$ such that 
    \begin{itemize} 
       \item $d_X \big(b, f(b)\big) \le 2 \delta$. 
       \item $d_X(b, \mu_B) \le d_X(\mu_B, f\mu_B) + 2\delta$.
    \end{itemize} 
    Let $f^i(b)$ be the transversal to $f^i(X)$ and let $\mu_O$ be the
    associated clean marking on $S$. The correct constant will be $R_O =
    R_I+2M_3$.

    Now suppose $X$ is a non-annular domain. Let $F = f^{L_X + 1} : X \to
    X$ be the first return map of $f$ to $X$. Set \[ R_I' = N R_I + 2M_3.
    \] Let $\nu_B = \Pi_X(\mu_B)$ and $\nu_I = \Pi_X(\mu_I)$ be
    respectively the induced markings of $\mu_B$ and $\mu_I$ on $X$. We will
    regard $\nu_B$ as the base marking in $\operatorname{Mark}(X)$. Since
    $\mu_I \in \wfix_{R_I}(f)$, for any $Z \subseteq X$,
    \begin{align*} 
      d_Z(\nu_I, F\nu_I) 
      & \le d_Z(\mu_I, F\mu_I) + 2M_3 \\
      & = d_Z(\mu_I, f^{L_X+1}\mu_I) + 2M_3 \\
      & \le \sum_{i=0}^{L_X} d_Z(f^i\mu_I, f^{i+1}\mu_I) + 2M_3\\
      & = (L_X+1) \, d_Z(\mu_I, f\mu_I) + 2M_3\\
      & \le (L_X+1)R_I + 2M_3 \\
      & <  R_I'.
    \end{align*}
    In other words, $\nu_I \in \wfix_{R_I'}(F)$. By Equation
    \eqref{eqInduced}, we have 
    \begin{align*}
      d_X(\mu_B, \mu_I) 
      & \le d_X(\mu_B, \nu_B) + d_X(\nu_B, \nu_I) +
      d_X(\nu_I, \mu_I) \\
      & \le d_X(\nu_B, \nu_I) + 2M_3
    \end{align*}
    Using above inequality and the fact that $X \in \Omega(\mu_I,R_I)$,
    we obtain
    \begin{align*}
      d_X(\nu_B, \nu_I) 
      & \ge d_X(\mu_B, \mu_I) - 2M_3 \\
      & > N \max_{0\le i \le L_X} \big\{ d_{f^i(X)}(\mu_B, f\mu_B) \big\} +
      NR_I
      + \Theta - 2M_3 \\
      & > \sum_{i=0}^{L_x} d_{f^{-i}(X)}(\mu_B, f\mu_B) + NR_I + \Theta - 2M_3 \\
      & = \sum_{i=0}^{L_X} d_X(f^i \mu_B, f^{i+1}\mu_B) + NR_I + \Theta - 2M_3 \\
      & \ge d_X(\mu_B, f^{L_X+1}\mu_B) + NR_I + \Theta - 2M_3 \\
      & \ge d_X(\nu_B, F\nu_B) + NR_I + \Theta - 4M_3 \\
      & \ge d_X(\nu_B, F\nu_B) + R_I'.
    \end{align*}
    In other words, $X \in \Omega(\nu_I, R_I', F)$. We may apply
    \propref{propBaseStep1}, treating $X$ as the whole surface. This gives
    a marking $\nu_O$ on $X$ and a constant $R_O' \ge R_I'$ depending only
    on $R_I'$ (hence $R_I$) such that 
    \begin{itemize}

      \item[(P1)] For any $Z \subseteq X$, 
        \begin{equation}\label{eqBS2a}
          d_Z(\nu_O, F\nu_O) \le R_O'.
        \end{equation}

      \item[(P2)] $X \notin \Omega(\nu_O, R_O', F)$, meaning
        \begin{equation}\label{eqBS2b}
          d_X(\nu_B, \nu_O) \le d_X(\nu_B, F\nu_B) + R_O'. 
        \end{equation}
        %$\Omega(\nu_O, R_O, F) \subsetneqq \Omega(\nu_I, R_I', F)$.
    \end{itemize} 
    The action of $f$ induces a marking $f^i \nu_O$ on each
    $f^i(X)$. We complete \[ c \cup \displaystyle \bigcup_{i=0}^{L_X} f^i \nu_O
    \] to a marking $\mu_O$ on $\s$ by extending $\mu_I$ to the remaining
    complements and the curves in $c$. In this case, set $R_O = R_O' +
    2M_3$. 
    
    Now, for $X$ is either a curve or a non-annular domain, let $\mu_O$ and
    $R_O$ be the appropriate marking and constant. We will show $\mu_O$ and
    $R_O$ satisfy properties (Q1), (Q3), (Q4) and (Q5). Let's consider the
    following analyses.  
    \begin{itemize}
    
      \item[(c1)] If $Z = \s$, by assumption, $\s \notin \Omega(\mu_I, R_I)$,
      so \[ d_\s (\mu_B, \mu_I) \le d_\s(\mu_B, f\mu_B) + R_I. \] Since $X$
      is a domain of a geodesic in $H(\mu_B, \mu_I)$, we have
      \begin{align*}
        d_\s(\mu_B, \mu_O) 
        & \le d_\s(\mu_B, c) + 2 \\
        & \le d_\s(\mu_B, \mu_I) + M_3 + 2 \\
        & \le d_\s(\mu_B, f\mu_B) + R_I + M_3 +2 \\ 
        & \le d_\s(\mu_B, f\mu_B) + R_O.
      \end{align*}
      In particular, $S \notin \Omega(\mu_O,R_O)$. As in (b1) of case II in
      \propref{propBaseStep1}, we also have \[ d_\s(\mu_O, f\mu_O) \le 4.
      \]

      \item[(c2)] If $Z \ne \s$ but some curve of $c$ crosses $Z$, then the
      same argument of (b2) of \propref{propBaseStep1} applies to give \[
      d_Z(\mu_B, \mu_O) \le d_Z(\mu_B, \mu_I) + M_3 + 4, \] and \[
      d_Z(\mu_O, f\mu_O) \le 8. \]

      \item[(c3)] If $Z$ is a subsurface of some component domain of $(\s,
      c)$ on which $\mu_O$ is induced from $\mu_I$ (this includes the
      possibility that $Z$ is a curve in $c$ when $X$ is not a curve),
      then, as in (b3) of \propref{propBaseStep1}, \[ d_Z(\mu_B, \mu_O) \le
      d_Z(\mu_B, \mu_I) + M_3, \] and \[ d_Z(\mu_O, f\mu_O) \le R_I + 2M_3
      \le R_O. \]

      \item[(c4)] If $X$ is a curve, then by construction 
      \begin{align*}
         d_X(\mu_B, \mu_O) 
         & \le d_X(\mu_B, b) + M_3 \\ 
         & \le d_X(\mu_B, f\mu_B) + 2\delta + M_3 \\
         & \le d_X(\mu_B, f\mu_B) + R_O,
      \end{align*}
      and
      \begin{align*}
         d_X(\mu_O, f\mu_O) 
         & \le d_X(b, f(b)) + 2M_3 \\
         & \le 2 \delta + 2M_3 \\ 
         & \le  R_O.
      \end{align*}
      
      If $X$ is non-annular,
      and $Z \subseteq X$, then it follows from \eqref{eqBS2a} that
      \begin{align*}
         d_Z(\mu_O, f\mu_O) 
         & \le d_Z(\nu_O, F\nu_O) + 2M_3 \\
         & \le R_O' + 2M_3 \\ 
         & = R_O. 
      \end{align*}
        
      Finally, \eqref{eqBS2b} yields
      \begin{align}\label{eqBS2d}
        d_X(\mu_B, \mu_O) 
        & \le d_X(\nu_B, \nu_O) + 2 M_3 \nonumber \\
        & \le d_X(\nu_B, F\nu_B) + R_O'+2M_3 \nonumber \\
        & \le d_X(\mu_B, f^{L_X+1}\mu_B) + R_O' + 4M_3 \nonumber \\
        & \le N \max_{0 \le i \le L_X} \big\{ d_{f^i(X)}(\mu_B, f\mu_B) \big\}
        + NR_O + \Theta.
      \end{align}

      One consequence here is that, whether or not $X$ is a curve, $X
      \notin \Omega(\mu_O, R_O)$.

      \item[(c5)] If $X$ is a curve and $0 < i \le L_X$, since both $\mu_O$ and
      $f\mu_O$ contain $f^i(b)$ (as a transversal), they are $M3$--close to
      $f^i(b)$ in $\mathcal{C}\big(f^i(X)\big)$. Hence \[ d_{f^i(X)} (\mu_O,
        f\mu_O) \le
        d_{f^i(X)}(f^i(b), f^i(b)) + 2M_3
      \le R_O. \] Furthermore, 
      \begin{align*}
         d_{f^i(X)}(\mu_B, \mu_O) 
         & \le d_{f^i(X)}\big(\mu_B, f^i(b)\big) + M_3 \\
         & \le d_X(f^{-i}\mu_B, b) + M_3 \\
         & \le d_X(\mu_B, b) + d_X(\mu_B, f^{-i}(\mu_B)) + M_3 \\
         & \le d_X(\mu_B, b) + \sum_{j=0}^{i-1} d_X(f^{-j}\mu_B,
         f^{-(j+1)}\mu_B) + M_3 \\ 
         & = d_X(\mu_B, f\mu_B) + 2\delta + \sum_{j=0}^{i-1}
         d_{f^{j+1}(X)}(\mu_B, f\mu_B) + M_3 \\
         & = \sum_{j=0}^{i} d_{f^j(X)}(\mu_B, f\mu_B) + 2\delta + M_3
         \\
        & \le N \max_{0\le j\le L_X} \big\{ d_{f^j(X)}(\mu_B, f\mu_B)\big\} +
          R_O.
      \end{align*}
      
      If $X$ is non-annular, and $Z \subseteq f^i(X)$, $0< i \le L_X$, then \[
      d_Z(\mu_O, f\mu_O) = d_Z(f^i\nu_O, f^i\nu_O) + 2M_3 \le R_O. \] 
        
      For $f^i(X)$, $0< i \le L_X$, then 
      \begin{align*}
        d_{f^i(X)}(\mu_B, \mu_O) 
        & \le d_{f^i(X)}(\mu_B, f^i\nu_O) + M_3 \\
        & = d_X(f^{-i}\mu_B, \nu_O) + M_3\\
        & \le d_X(\mu_B, \nu_O) + \sum_{j=0}^{i-1} d_X(f^{-j} \mu_B,
        f^{-(j+1)}\mu_B)\\
        \text{By } \eqref{eqBS2d} \quad
        & \le 2N \max_{0\le j\le L_X} \big\{ d_{f^j(X)}(\mu_B, f\mu_B)\big\} +
        NR_O + \Theta
      \end{align*} 
     
     An consequence here is that $f^i(X) \notin \Omega(\mu_O, R_O)$.
   
    \end{itemize}

    Together from (c1) to (c5), we have shown that $\mu_O \in
    \wfix_{R_O}(f)$. Property (Q3) is verified in cases (c4) and (c5). To
    see (Q5), if $Z \in \Omega(\mu_O, R_O)$, then $Z$ is either of case
    (c2), (c3), or $Z \subsetneqq f^i(X)$, for some $i = 0, \ldots, L_X$.
    In (c2) or (c3), since \[ d_Z(\mu_B, \mu_O) \le d_Z(\mu_B, \mu_I) +
    M_3+ 4, \] it follows that %(see \remref{remBaseStep1}) 
    \[ Z \in
    \Omega(\mu_O, R_O) \quad \Longrightarrow \quad Z \in \Omega(\mu_I,
    R_I). \] To see (Q4), we use \lemref{lemTimeOrder} on the assumption
    $X <_t f^{-i}(Z)$ to obtain \[ d_{f^{-i}(Z)}( \mu_B, \partial X) \le M_1.
    \] Since $\mu_O$ contains $c = \bigcup_{j} \partial f^j(X)$ as base
    curves, we have
    \begin{align*}
      d_Z(\mu_B, \mu_O) 
      & = d_{f^{-i}(Z)}(f^{-i}\mu_B, f^{-i}\mu_O) \\
      & \le d_{f^{-i}(Z)}(f^{-i}\mu_B, \mu_B) + d_{f^{-i}(Z)}(\mu_B,
      f^{-i}\mu_O) \\
      & \le N \max_{0 \le i \le L_Z} \big\{ d_{f^i(Z)}(\mu_B, f\mu_B) \big\} +
      d_{f^{-i}(Z)}(\mu_B, \partial X) + M_3 \\
      & \le N \max_{0 \le i \le L_Z} \big\{ d_{f^i(Z)}(\mu_B, f\mu_B) \big\} +
      M_1 + M_3 
    \end{align*}
    Therefore, $Z \notin \Omega(\mu_O, R_O)$. This concludes (Q4) and the
    proof in the case $n=1$.

    In the case that $n > 1$, the proof is essentially the same. By
    re-indexing if necessary, we may assume $X_1, \ldots X_m$, $m \le n$,
    are non-annular domains, and $X_{m+1}, \ldots, X_n$ are curves. The
    assumption that $\mathcal{U}_1, \ldots, \mathcal{U}_n$ are pairwise
    disjoint allows us to apply the proof in the case $n=1$ to all $X_j$'s
    simultaneously. More precisely, for each $1 \le j \le m$, let
    $\nu_{O,j}$ be the marking in $X_j$ coming from the proof in case
    $n=1$. Similarly, for $m+1 \le j \le n$, let $b_j$ be the transversal
    curve to $X_j$ coming from the proof in the case $n=1$. Let $c = \cup_j
    c_j$ where $c_j = \partial X_j \cup \cdots \cup f^{L_{X_j}}(X_j)$. The
    set \[ c \cup \left( \bigcup_{j=1}^{m} \bigcup_{i=0}^{L_{X_j}}
      f^i\nu_{O,j} \right) \cup \left( \bigcup_{i=1}^{L_{X_j}}
      \bigcup_{j=m+1}^n f^i(b_j) \right)
    \] can be extended to marking $\mu_O$ by extending $\mu_I$ to the
    remaining complements and the curves in $c_1, \ldots, c_m$. Let $R_O$
    be the maximum of the two constants from the proof in the case $n=1$
    (one constant for $X$ a curve and the second for $X$ a non-annular
    domain). The proof that $\mu_O$ and $R_O$ satisfy the desired
    properties (Q1) to (Q5) is the same as the proof in for $n=1$.
  \end{proof}
  
  \subsection{Reducing complexity}\label{secIteration}

  In this section, we show how to use \propref{propBaseStep1} and
  \propref{propBaseStep2} to construct $R$ and $\mu$ for
  \thmref{thmFiniteOrder}. 

  From now on, a pair $(\mu, R)$ will always mean $\mu \in \wfix_R(f)$.
    
  \begin{proposition}[Reducing complexity]\label{propReduction}

    Let $f \in \mcg$ be of finite order. Let $R_I$ and $\mu_I$ be as in
    \propref{propBaseStep1}. Suppose $\Omega(\mu_I, R_I) \ne \emptyset$.
    There exists $R_O$, depending only on $R_I$, and $\mu_O \in
    \wfix_{R_O}(f)$ such that \[ \xi(\mu_O, R_O) < \xi(\mu_I, R_I). \]

  \end{proposition}

  \begin{proof}

    If $\s \in \Omega(\mu_I, R_I)$, then \propref{propBaseStep1} produces
    $(\mu_O, R_O)$ such that $R_O$ depends only on $R_I$ and $\s \notin
    \Omega(\mu_O, R_O)$, hence $\xi(\mu_O, R_O) < \xi(\s) = \xi(\mu_I, R_I)$. 

    Now suppose $\s \notin \Omega(\mu_I, R_I)$. Choose a maximal element
    $X_1 \in \Omega(\mu_I, R_I)$. This in particular means $X_1$ has
    maximal complexity over all elements of $\Omega(\mu_I, R_I)$. Set $
    \mathcal{U}_1 = X_1 \cup \cdots \cup f^{L_{X_1}}(X_1).$  Consider the
    maximal complexity of the elements in $\Omega(\mu_I, R_I)$ supported on
    $\s \setminus \mathcal{U}_1$. If this complexity is strictly less than
    $\xi(X_1)$, then we stop. If this complexity is not strictly less than
    $\xi(X_1)$, then we may choose $X_2$ of maximal order in $\Omega(\mu_I,
    R_I)$ supported on $\s \setminus \mathcal{U}_1$ such that
    $\xi(X_2)=\xi(X_1)$. Set $\mathcal{U}_2 = X_2 \cup \cdots \cup
    f^{L_{X_2}}(X_2)$. In this case, $\mathcal{U}_1$ and $\mathcal{U}_2$
    are disjoint. Now we repeat this process by considering the maximum
    complexity of the elements in $\Omega(\mu_I, R_I)$ supported on $\s
    \setminus ( \mathcal{U}_1 \cup \mathcal{U}_2 )$. Continuing this way,
    we eventually \emph{exhaust} $\s$ by a sequence \[ \mathcal{U}_1,
    \mathcal{U}_2, \ldots, \mathcal{U}_n \] in the following sense:
    \begin{itemize}
        
      \item For each $i$, the set $\mathcal{U}_i$ is a disjoint union of
        subsurfaces of \s of the form \[ \mathcal{U}_i =  X_i \cup \cdots
        \cup f^{L_{X_i}}(X_i), \] with $\xi(X_i) = \xi(X_1)$.

      \item The sets $\mathcal{U}_1, \ldots, \mathcal{U}_n$ are pairwise
        disjoint.

      \item The maximal complexity of the bad domains in $\Omega(\mu_I,
        R_I)$ supported on \[ \s \setminus (\mathcal{U}_1 \cup \dots \cup
        \mathcal{U}_n) \] is strictly less than $\xi(X_1)$.
        
    \end{itemize}
    
    Note that the exhaustion sequence has length $n$ which is bounded
    uniformly by a constant depending only on $\s$. Denote by \[ c_i =
    \partial X_i \cup \cdots \cup \partial f^{L_{X_i}}(X_i). \] By
    assumption, $\mathcal{U}_1, \ldots, \mathcal{U}_n$ are pairwise
    disjoint, so we can apply \propref{propBaseStep2} to construct a pair
    $(\mu_O, R_O)$ with $\mu_O$ containing $c_1 \cup \cdots \cup c_n$ as
    base curves and $R_O$ depending only on $R_I$. By properties (Q3),
    (Q4), and (Q5) of \propref{propBaseStep2}, if $Z \in \Omega(\mu_O,
    R_O)$, then either
    \begin{itemize}
      
      \item[(i)] $Z$ intersects some curve in $c_i$ for some $i$. (See 
      case (c2) in the proof of \propref{propBaseStep2})
        
      \item[(ii)] $Z$ is supported on $\s \setminus (\mathcal{U}_1 \cup
      \dots \cup \mathcal{U}_n).$ (See case (c3) in the proof of
      \propref{propBaseStep2}.)

      \item[(iii)] $Z \subsetneqq f^j(X_i)$, for some $0 \le i \le n$ and
        $0 \le j \le L_{X_i}$. 
        
    \end{itemize} Immediately, case (iii) has $\xi(Z) < \xi(X_1)$. Recall
    that for either case (i) or (ii), \[ Z \in \Omega(\mu_O, R_O) \qquad
    \Longrightarrow \qquad Z \in \Omega(\mu_I, R_I).  \] Since $\s$ is
    exhausted by assumption, case (ii) also means $\xi(Z) < \xi(X_1)$.
    Lastly, suppose $Z$ is of case (i). Choose the minimal index $i$ such
    that $Z$ intersects a curve in $c_i$. In other words, $Z$ is supported
    on \[ \s \setminus (\mathcal{U}_1 \cup \cdots \cup \mathcal{U}_{i-1}),
    \] and $Z$ interlocks $f^j(X_i)$, for some $j$. Our choice of $X_i$ has
    maximal order among the bad domains supported on \[ \s \setminus
    (\mathcal{U}_1 \cup \cdots \cup \mathcal{U}_{i-1}). \] Therefore, if
    $\xi(Z) = \xi(X_1) = \xi(X_i) $, then $X_i  <_t f^{-j}(Z)$ in $H(\mu_B,
    \mu_I)$. Property (Q5) of \propref{propBaseStep2} guarantees that such
    domains do not appear in $\Omega(\mu_O, R_O)$. Thus, any $Z$ of case
    (i) must also have $\xi(Z) < \xi(X)$.  
  \end{proof}

  Let $R_1$ be the minimal constant satisfying \lemref{lemFixedPoints}.

  \begin{corollary}[Termination]\label{corTermination}

    There exists $R \ge R_1$ depending only on $\s$ such that any finite order
    $f \in \mcg$ has a $\mu \in \wfix_R(f)$ satisfying $\Omega(\mu,R)
    = \emptyset$.
    
  \end{corollary}

  \begin{proof}

    Let $\mu_1 \in \wfix_{R_1}(f)$. If $\Omega(\mu_1, R_1) = \emptyset$,
    then we are done. If not, then applying \propref{propReduction}
    iteratively yields a sequence of pairs $(\mu_1, R_1)$, $(\mu_2, R_2)$,
    $\ldots$ such that \begin{itemize}
      
      \item $\mu_{i+1} \in \wfix_{R_{i+1}}(f)$, where $R_{i+1}$ depends
      only on $R_i$.

      \item $\xi_{i+1} \lneqq \xi_i$, where $\xi_i = \xi(\mu_i, R_i)$.

    \end{itemize}
    Since $\xi_i$ corresponds to the maximum complexity over elements in
    $\Omega(\mu_i, R_i)$, and that $\xi_i$'s are strictly decreasing, we
    must eventually reach a pair $(\mu_n, R_n)$ for which $\xi_n = -2$,
    i.e. $\Omega(\mu_n, R_n) = \emptyset$. Moreover, since elements of
    $\Omega(\mu_i, R_i)$ come from subsurfaces of $\s$, $\xi_1 \le  \xi(\s)
    = 3g- 3+b$. This gives a bound on $n \le 3g-1+b$, and therefore $R_n$
    depends only on $\s$.
  \end{proof}

  This concludes the proof of \thmref{thmFiniteOrder}.
    
\section{L.B.C.~property for reducible mapping classes}
  
  \label{secReducible} 

  In this section, we prove L.B.C.~property for (infinite-order) reducible
  elements of \mcg. We would like to use an induction argument on
  subsurfaces. To do so, we make the following observation. 
  
  Let $f \in \mcg$ be reducible with canonical reducing system $\sigma$
  (see \secref{secCanonical}). If $\mu \in \mg$ is a marking containing
  $\sigma$ as base curves, then so does $f\mu$. This means that any
  hierarchy $H(\mu,f\mu)$ decomposes into geodesics supported on component
  domains of $(\s, \sigma)$. For such a marking $\mu$, we can control each
  component domain independently, allowing the arguments of \thmref{thmPA}
  and \corref{corFiniteOrder} to pass through to subsurfaces. This inspires
  the definition of a good marking for $f$ (\defref{defGoodMarking}). We
  construct a finite collection of good markings and prove
  \thmref{introthmReducible} of the introduction. The induction argument on
  subsurfaces using good markings appears in \propref{propBaseCase}.
  Finally, \corref{corReducible} combines the finiteness and the induction
  argument to finish the proof of L.B.C.~property for reducible elements of
  \mcg. 
 
  \begin{definition}[Good marking] \label{defGoodMarking}
  
    Let $f \in \mcg$ be an infinite-order reducible element and let
    $\sigma$ be its associated canonical reducing system. We say a marking
    $\mu \in \mg$ is a \emph{good marking} for $f$ if $\sigma \subseteq
    \base(\mu)$. 
  
  \end{definition}
  
  Up to homeomorphisms of \s, there are only finitely many multicurves on
  \s. If $f \in \mcg$ is reducible and $f=\omega^{-1}g\omega$, then
  $\omega(\sigma_f)=\sigma_g$, where $\sigma_f$ and $\sigma_g$ are
  canonical reducing system for $f$ and $g$ respectively. We fix a
  representative for each homeomorphism type of a multicurve. Further, for
  each representative multicurve $\sigma$, we complete $\sigma$ into a
  finite set of markings representing each homeomorphism type of a marking
  containing $\sigma$ as base curves. Let $\mathcal{M}$ the collection of
  all such \emph{representative markings}, one from each homeomorphism
  type. The following is \thmref{introthmReducible} of the introduction.

  \begin{theorem}\label{thmReducible} 
          
    There exists $a \in \mcg$ such that $a^{-1}fa$ has a good marking in
    $\m$, and \[ d_{\mg}(\mu_B, a\mu_B) \prec d_{\mg}(\mu_B, f\mu_B). \]
    Furthermore, if $f$ and $g$ are conjugate and $b^{-1}gb=a^{-1}fa$ with
    $d_{\mg}(\mu_B, b\mu_B) \prec d_{\mg}(\mu_B, g\mu_B)$, then we may choose the
    same good marking in $\m$ for $b^{-1}gb$.
    
  \end{theorem}

  Before proving \thmref{thmReducible}, we set up some notations. Fix a
  multicurve $\sigma$ on \s. Let $\stab(\sigma) < \mcg$ be the subgroup
  stabilizing $\sigma$ as a set:
  \[ \stab(\sigma) : \{\,\, h\in \mcg \,\,\, : h(\sigma) = \sigma \,\,\}.
  \] The action of $\stab(\sigma)$ on the complementary components of $\s
  \setminus \sigma$ induces an exact sequence:
  \begin{equation} \label{eqStab} 1 \longrightarrow \stab_0(\sigma)
  \longrightarrow \stab(\sigma) \stackrel{\pi}{\longrightarrow}
  \text{Finite Group} \longrightarrow 1.
  \end{equation}
  Consider the kernel $\stab_0(\sigma)$ of the above sequence. If $f \in
  \stab_0(\sigma)$, then $f$ acts on each complementary component $Y
  \subset \s \setminus \sigma$. In other words, $f$ defines an element
  $f|_Y \in \mathcal{MCG}_0(Y)$ for each $Y \subset \s \setminus \sigma$,
  where $\mathcal{MCG}_0(Y) \subset \mathcal{MCG}(Y)$ is the subgroup
  fixing $\partial Y$. We have an exact sequence: 
  \begin{equation} \label{eqStab0}
     1 \longrightarrow T_\sigma \longrightarrow \stab_0(\sigma)
     \longrightarrow \prod_{Y \in \s \setminus \sigma} \mathcal{MCG}_0(Y)
     \longrightarrow 1,
  \end{equation}
  where $T_\sigma$ is a free abelian group with basis the Dehn twists along
  curves in $\sigma$. By the classification theorem, $f|_Y$ is either
  pseudo-Anosov or has finite order. We say an element $f \in
  \stab_0(\sigma)$ is \emph{pure} if $f|_Y$ is either pseudo-Anosov or the
  identity on $Y$. The order of the finite group in \eqref{eqStab} is
  bounded by a constant $N$ depending only on \s. Thus, for any $\sigma$
  and any $f  \in \stab(\sigma)$, $f^N \in \stab_0(\sigma)$. Moreover,
  since there are only finitely many subsurfaces of \s up to homeomorphism,
  one can choose a constant for \corref{lemOrderBound} which works for \s
  and all subsurfaces of \s. Thus, there exists some universal power
  $N=N(S)$ depending only \s, such that for any reducible mapping element
  $f \in \mcg$, $f^N$ is pure.  
  
  We can characterize the canonical reducing system for a reducible mapping
  class as follows. Suppose $f \in \stab(\sigma)$. Let $g=f^N$ be pure.
  Then $\sigma=\sigma_f$ is the canonical reducing system for $f$ if for
  any $\alpha \in \sigma$, one of the following holds.
  \begin{itemize}
     
     \item[(H1)] There exists a domain $Y$ in $\s
     \setminus \sigma$ such that $\alpha$ is a boundary component of $Y$
     and $g|_Y$ is pseudo-Anosov on $Y$.
     
     \item[(H2)] There exists a domain $Z \subset \s$ such that $g|_Z$
     is a non-zero power of a Dehn twist along $\alpha$.       
  \end{itemize}
  To see this, let $\alpha \in \sigma$ be such that condition (H1) does not
  hold. Then $\alpha$ must bound two (not necessarily distinct) components
  $X$ and $Y$ of $\s \setminus \sigma$ such that $g|_X$ and $g|_Y$ are both
  the identity. In this case, let $Z=X \cup Y \cup \alpha$. Then $g|_Z$ is a
  non-zero power of Dehn twist along $\alpha$. Otherwise, the first return
  map of $f$ to $Z$ is of finite order, and one can thus obtain a smaller
  reducing system for $f$ by removing $\alpha$, contradicting minimality of
  $\sigma$. Note that this proof also implies that if $\sigma$ is a
  canonical reducing system for $f$, then $\sigma$ is also the canonical
  reducing system for any power of $f$. 
  
  In the case of (H2), it follows that for any $n \in \nn$ and any $v \in
  \mathcal{C}(\alpha)$, \[ d_\alpha\big(v, g^n(v)\big) \ge |n|. \] Compare
  this with \lemref{lemPA}. Since there are only finitely many domains of
  \s up to homeomorphism, we can choose $N_0$ depending only on \s such
  that the following hold. Let $M_2$ be the constant of
  \lemref{lemLargeLink}. For any multicurve $\sigma$ and any $g \in
  \stab_0(\sigma)$, let $Y$ be either a component of $\s \setminus \sigma$
  on which $g$ is pseudo-Anosov, or $Y$ is a curve in $\sigma$ such that
  property (H2) holds, then for any $n \ge N_0$ and any $v \in
  \mathcal{C}(Y)$, \[d_Y \big( v,g^n (v) \big) \ge M_2.\] 
  
  For any $g \in \stab_0(\sigma)$, Equation \eqref{eqUndistorted} has the
  following consequence. If $\mu \in \mg$ is a good marking for $g$, then
  \begin{equation} \label{eqUndistorted2} d_{\mg}(\mu, g\mu) \asymp \sum_{X
      \subset \s \setminus \sigma} d_{\Mark(X)}
     (\mu, g|_X \mu) + \sum_{\alpha \in \sigma} d_\alpha(\mu, g\mu),
  \end{equation}
  where
  \[ d_{\Mark(X)}(\mu, g|_X \mu) := d_{\Mark(X)} \big( \Pi_X(\mu), g|_X
  \Pi_X(\mu) \big) \asymp d_{\Mark(X)} \big( \Pi_X(\mu), \Pi_X(g\mu) \big).
  \] 
  
  We will say an element $g \in \stab_0(\sigma)$ \emph{does not twist}
  along $\alpha \in \sigma$ if for any $v \in \mathcal{C}(\alpha)$, \[
  \lim_{n \to \infty} \frac{d_\alpha(v, g^n(v))}{n} = 0. \] In this case,
  $d_\alpha(\mu,g\mu) \prec 1$, and one can ignore the second summand on
  the right hand side of Equation \eqref{eqUndistorted2}.
   
  \begin{proof}[Proof of \thmref{thmReducible}]
    
    The set $\m$ is finite and each conjugacy class of $\mcg$ has a good
    marking in $\m$ by construction. Let \[ C = \max\{\,\, d_{\mg}(\mu_B,\mu)
    \,\,\, : \,\,\, \mu \in \m \,\,\}. \] 
    
    Let $f \in \mcg$ be reducible with canonical reducing system $\sigma_f$
    and set $F = f^{N_0}$. Let $\mu \in \m$ be arbitrary. By our definition
    of $N_0$, for any $\alpha \in \sigma_f$, $\alpha$ is either a boundary
    curve of a domain $Y$ such that $d_Y(\mu, F\mu) \ge M_2$, or
    $d_\alpha(\mu, F\mu) \ge M_2$. Thus, by \lemref{lemLargeLink}, $\alpha$
    is either a domain for a geodesic in $H(\mu,F\mu)$ or is a boundary
    curve of a domain for a geodesic in $H(\mu, F\mu)$. A consequence is
    that, for any $Y \subset \s$ that intersects a curve $\alpha \in
    \sigma_f$, by choosing a hierarchal slice containing $\alpha$ and using
    \lemref{lemTriangle}, we have 
    \begin{align} \label{eqTri}
      d_Y(\alpha, F\mu) \prec d_Y(\mu, F\mu).
    \end{align}

    Now let $\mu'$ be a marking extension of $\sigma_f$ relative to $F\mu$.
    By \lemref{lemMarkingExtension}, for any component domain $Y$ of $(\s,
    \sigma_f)$, $d_Y(\mu',F\mu)$ is uniformly bounded. On the other hand,
    if $Y \subseteq \s$ is any domain that intersects some curve $\alpha
    \in \sigma_f$, then by \eqref{eqTri} and the fact that $\sigma_f
    \subseteq \base(\mu')$, we have
    \begin{align*}
      d_Y(\mu', F\mu) \le d_Y(\alpha, F\mu) + 2 \prec d_Y(\mu, F\mu) 
    \end{align*}
    By ranging over all $Y \subseteq \s$ on which $d_Y(\mu', F\mu)$ is
    sufficiently large, we obtain 
    \[ d_{\mg}(\mu', F\mu) \prec d_{\mg}(\mu, F\mu). \]
    This implies:
    \begin{align*} 
      d_{\mg}(\mu, \mu') 
      &\le d_{\mg}(\mu, F\mu) + d_{\mg}(F\mu, \mu') \\
      &\prec d_{\mg}(\mu, F\mu) \\
      &\prec d_{\mg}(\mu, f\mu).
    \end{align*}
    Let $\sigma$ be the representative multicurve for $\sigma_f$ and let
    $\m(\sigma) \subset \m$ be the subset of markings containing $\sigma$
    as base curves. Now choose a marking $\mu'' \in \m(\sigma)$ such that
    there exist $a \in \mcg$ with $a(\mu'') = \mu'$. By construction,
    $a^{-1}fa$ has canonical reducing system $a^{-1}(\sigma_f) = \sigma
    \subseteq \base(\mu'')$. We also have
    \begin{align*}
      d_{\mg}(\mu_B, a\mu_B) 
      & \le d_{\mg}(\mu_B, a \mu'') + d_{\mg}(a\mu'', a\mu_B) \\ 
      & \le d_{\mg}(\mu_B, \mu') + C \\ 
      & \le d_{\mg}(\mu_B, \mu) + d_{\mg}(\mu, \mu') + C\\ 
      & \prec d_{\mg}(\mu, f\mu) + 2C\\
      & \le d_{\mg}(\mu, \mu_B) + d_{\mg}(\mu_B, f\mu_B) + d_{\mg}(f\mu_B, f\mu) + 2C\\
      & \le d_{\mg}(\mu_B, f\mu_B) + 4C.
    \end{align*}
    If $g \in \mcg$ is conjugate to $f$, then $\sigma$ would also be the
    representative multicurve for $\sigma_g$. Our construction produces an
    element $b \in \mcg$ such that $b^{-1}gb$ has a good marking in
    $\m(\sigma)$ and $d_{\mg}(\mu_B, b\mu_B) \prec d_{\mg}(\mu_B, g\mu_B)$. Since
    any marking in $\m(\sigma)$ is a good marking for $b^{-1}gb$, including
    $\mu''$, and $\m(\sigma)$ is a finite set, the second statement follows.
  \end{proof}  
  
  \begin{proposition}\label{propBaseCase}
    
    Suppose $f, g \in \mcg$ are two conjugate infinite-order reducible
    mapping classes with the same canonical reducing system $\sigma$. Let
    $\mu$ be a good marking for $f$ and $g$. Then there exist a constant
    $K_\mu$ and $\omega \in \mcg$ such that $\omega$ is a conjugator for
    $f$ and $g$, and \[ d_{\mg}(\mu,\omega\mu) \prec K_\mu \big ( d_{\mg}(\mu,f\mu)
    + d_{\mg}(\mu,g\mu) \big).\] 
  \end{proposition}

  \begin{proof}
   
    Elements of $\stab_0(\sigma)$ are easier to handle, but a conjugator
    for $f^n$ and $g^n$ is not a conjugator for $f$ and $g$. Thus we cannot
    apply the results of \corref{corFiniteOrder} and \thmref{thmPA}, to
    powers of $f$ and $g$. We must deal with the issue of permuting
    subsurfaces in the proof of \propref{propBaseCase}. Fix a finite
    collection $\mathcal{P} \subset \stab(\sigma)$ such that
    $\pi(\mathcal{P})$ in the exact sequence \eqref{eqStab} is onto. Let
    \[P = \max \{\,\, d_{\mg}(\mu, a\mu) \,\,\,:\,\,\, a \in
    \mathcal{P} \,\,\}.\] Choose $a \in \mathcal{P}$ such that $f$ and
    $g' = aga^{-1}$ have the following properties: 
    \begin{itemize}
      
      \item[(i)] $\pi(f)=\pi(g')$.
        
      \item[(ii)] For any $X$ in $\s \setminus \sigma$, the first return
        map to $X$ of $f$ and $g'$ are conjugate.

    \end{itemize}
    If the proposition holds for $f$ and $g'$, say, there exist $K_\mu$ and
    $\omega \in \mcg$ such that $f \omega = \omega g'$ and \[d_{\mg}(\mu,
    \omega' \mu) \prec  K_\mu\big(d_{\mg}(\mu, f \mu) + d_{\mg}(\mu, g' \mu)\big),\]
    then $\omega a$ and $K_\mu + 2PK_\mu + P$ verify the proposition
    for $f$ and $g$. Clearly, $\omega a$ is a conjugator for $f$ and
    $g$. It remains to check: 
    \begin{align*} 
      d_{\mg}(\mu,\omega a \mu) 
      & \le d_{\mg}(\mu, \omega \mu) + d_{\mg}(\omega \mu, \omega a \mu) \\ 
      & \le d_{\mg}(\mu, \omega \mu) + d_{\mg}(\mu, a\mu) \\ 
      & \prec K_\mu \big( d_{\mg}(\mu, f \mu) + d_{\mg}(\mu, aga^{-1}
       \mu) \big) + P \\ 
       & \le K_\mu \big( d_{\mg}(\mu,f\mu) + d_{\mg}(\mu,g\mu) + 2P \big) + P.
       \end{align*}
    Thus, we may assume $f$ and $g$ already satisfy properties (i) and (ii)
    above. We will find a conjugator $\omega \in \stab_0(\sigma)$ for $f$
    and $g$. To do this, we will use properties (i) and (ii) and the
    induction hypothesis to build a conjugating element $\omega_Y \in
    \mathcal{MCG}_0(Y)$ for each component $Y$ in $\s \setminus \sigma$.
    Then we will choose an appropriate lift $\omega \in \stab_0(Y)$.
   
    Decompose the complementary components of $\s \setminus \sigma$ into
    orbits under the action $f$. Pick a representative from each orbit. Let
    $X_1$ be one such representative and consider the sequence of distinct
    complementary subsurfaces of $\s \setminus \sigma$ \[ X_1, X_2, \ldots,
    X_n \] such that $f(X_i) = X_{i+1}$ and $g(X_i) = X_{i+1}$, for $i=1,
    \ldots, n$ and $X_{n+1} = X_1$. Note that $n < N_0$. Set $f_i =
    f|_{X_i} : X_i \to X_{i+1}$, and similarly for $g_i$. Set the first
    return maps $F = f^{n+1}|_{X_1} \in \mathcal{MCG}(X_1)$ and $G =
    g^{n+1}|_{X_1} \in \mathcal{MCG}(X_1)$. The assumption is that $F$ and
    $G$ are conjugate in $\mathcal{MCG}(X_1)$. By \thmref{thmNTCMCG}, $F$
    and $G$ are either pseudo-Anosov or have finite order on $X_1$. Letting
    $\nu_1 = \Pi_{X_1}(\mu)$ be the induced marking on $X_1$, it follows
    from results of \corref{corFiniteOrder} and \thmref{thmPA} that there
    exist $\omega_1 \in \mathcal{MCG}_0(X_1)$ and $K_1=K_1(\nu_1, X_1)$
    such that $F \omega_1 = \omega_1 G$, and 
    \begin{align*} 
       d_{\Mark(X_1)}(\nu_1, \omega_1 \nu_1) 
       & \le K_1 \big( d_{\Mark(X_1)}(\nu_1, F\nu_1) + d_{\Mark(X_1)}(\nu_1,
       G\nu_1) \big) \\ 
       & \prec K_1 \big( d_{\mg}(\mu, F\mu) + d_{\mg}(\mu, G\mu) \big) \qquad
       (\text{By \eqref{eqUndistorted}}) \\ 
       & \prec K_1 \big( d_{\mg}(\mu, f^n\mu) + d_{\mg}(\mu, g^n\mu) \big) \\ 
       & \prec K_1 \big( d_{\mg}(\mu, f\mu) + d_{\mg}(\mu, g\mu) \big)  
    \end{align*}
    Using $f$ and $g$, we construct for each $X_i$ an element $\omega_i \in
    \mathcal{MCG}_0(X_i)$ such that $f_i \omega_i = \omega_{i+1} g_i$, for
    $i = 1, \ldots n$ and $n+1 = 1$. The element $\omega_1 \in
    \mathcal{MCG}_0(X_1)$ is defined. For each $i=1,\ldots n$, set \[
    \omega_{i+1} = f_i \cdots f_1 \omega_1 g_1^{-1} \cdots g_i^{-1}.  \] In
    particular, $\omega_{n+1} = F\omega_1G^{-1} = \omega_1$. 
    Let $\nu_i = \Pi_{X_i}(\mu)$. We have 
    \begin{align*}
       d_{\Mark(X_i)}(\nu_i, \omega_{i+1}\nu_i) 
       & = d_{\Mark(X_i)}(\nu_i, f_i \cdots f_1 \omega_1 g_1^{-1} \cdots
       g_i^{-1}\nu_i) \\ 
       & = d_{\Mark(X_i)}(\nu_i, f^i|_{X_1} \omega_1 g^{-i}|_{X_i} \nu_i) \\ 
       & = d_{\Mark(X_1)}(f^{-i}|_{X_i} \nu_i, \omega_1 g^{-i}|_{X_i} \nu_i) \\ 
       & \le d_{\Mark(X_1)}(f^{-i}|_{X_i} \nu_i, \nu_1) + d_{\Mark(X_1)}(\nu_1,
       \omega_1 g^{-i}|_{X_i} \nu_i) \\ 
       & = d_{\Mark(X_1)}(f^{-i}|_{X_i} \nu_i, \nu_1) +
       d_{\Mark(X_1)}(\omega_1^{-1} \nu_1, g^{-i}|_{X_i} \nu_i) \\ 
       & \le d_{\Mark(X_1)}(f^{-i}|_{X_i} \nu_i, \nu_1) +
       d_{\Mark(X_1)}(\omega_1^{-1} \nu_1, \nu_1) + d_{\Mark(X_1)}(\nu_1,
       g^{-i}\nu_i) \\ 
       & \prec d_{\mg}(\mu, f^i\mu) + d_{\mg}(\mu, g^i\mu) + d_{\Mark(X_1)}(\nu_1,
       \omega_1 \nu_1) \\
       & \prec K_1 \big( d_{\mg}(\mu, f\mu) + d_{\mg}(\mu, g\mu) \big)
    \end{align*}

    We do this for each orbit of complementary subsurfaces in $\s \setminus
    \sigma$, building for each $Y \subset \s \setminus \sigma$ an element
    $\omega_Y \in \mathcal{MCG}_0(Y)$. Consider any element $\omega \in
    \stab_0(\sigma)$ such that $\omega|_Y = \omega_Y$. Since twisting
    commute, any $\omega$ will satisfy $f \omega = \omega g$ by
    construction. Thus, we can choose a lift $\omega$ which does not twist
    along any curves in $\sigma$. Let $\{K_i\}$ be the constants associated
    to each orbit and let $K_\mu = \max \{K_i\}$. Using previous work and
    Equation \eqref{eqUndistorted2}, we obtain
    \begin{align*} 
      d_{\mg}(\mu, \omega\mu) 
       & \asymp \sum_{X \subset \s \setminus \sigma} d_{\Mark(X)}(\mu, \omega_X
       \mu) + \sum_{\alpha \in \sigma} d_\alpha(\mu, \omega \mu) \\ 
       & \prec \sum_{X \subset \s \setminus \sigma} K_\mu \big( d_{\mg}(\mu, f\mu)
       + d_{\mg}(\mu, g\mu) \big) \\ 
       & \prec K_\mu \big( d_{\mg}(\mu, f\mu) + d_{\mg}(\mu, g\mu) \big). 
       \qedhere   
    \end{align*}
  \end{proof}
  
  \begin{corollary}[L.B.C.~property for reducible mapping
    classes] \label{corReducible}

    If $f, g \in \mcg$ are conjugate reducible mapping classes of infinite
    order, then there is a conjugating element $\omega \in \mcg$ with \[
    |\omega| \prec |f| + |g|.  \]

  \end{corollary}
  
  \begin{proof}

    Let $\m$ be the set of representative markings and let $C$ be the
    constant bounding the diameter of $\m$. Let $K$ be the constant
    depending only on \s defined by \[ K = \max\{\,\, K_\mu \,\,\,:\,\,\,
    \mu \in \m \,\,\}, \] where $K_\mu$ is the constant associated to $\mu$
    in \propref{propBaseCase}. Suppose $f_1$ and $f_2$ are conjugate
    reducible mapping classes of infinite order. Let $a_1$ and $a_2$ be
    such that $a_if_ia_i^{-1}=g_i$ have a good marking $\mu \in \m$, and
    satisfying $d_{\mg}(\mu_B, a_i\mu_B) \prec d_{\mg}(\mu_B, f_i\mu_B)$. Then each
    \begin{align*}
      d_{\mg}(\mu_B, g_i\mu_B) 
      & \le 2d_{\mg}(\mu_B, a_i\mu_B) + d_{\mg}(\mu_B, f_i\mu_B) \\
      & \prec d_{\mg}(\mu_B, f_i\mu_B).
    \end{align*}
    By \propref{propBaseCase}, there exists $\omega \in \mcg$ such that
    $g_1\omega = \omega g_2$, and 
    \begin{align*}
      d_{\mg}(\mu,\omega\mu) 
      & \prec K_\mu \big( d_{\mg}(\mu,g_1\mu) + d_{\mg}(\mu,g_2\mu) \big) \\
      & \le K \big( d_{\mg}(\mu, g_1\mu) + d_{\mg}(\mu, g_2\mu) \big).  
    \end{align*}
    Hence, by the triangle inequality,
    \begin{align*}
      d_{\mg}(\mu_B,\omega\mu_B) 
      & \le d_{\mg}(\mu_B,\mu) + d_{\mg}(\mu,\omega\mu) + d_{\mg}(\omega\mu,\omega\mu_B) \\ 
      & \le 2C + K d_{\mg}(\mu,g_1\mu) + d_{\mg}(\mu,g_2\mu)) \\ 
      & \le 2C + K \big( 4C + d_{\mg}(\mu_B,g_1\mu_B) + d_{\mg}(\mu_B,g_2\mu_B) \big) \\ 
      & \prec d_{\mg}(\mu_B, g_1\mu_B) + d_{\mg}(\mu_B, g_2\mu_B)
    \end{align*}
    Set $\omega' = a_1^{-1}\omega a_2$. Then $\omega'$ is a conjugator for $f_1$
    and $f_2$, and 
    \begin{align*} 
      d_{\mg}(\mu_B, \omega'\mu_B) 
      & \le d_{\mg}(\mu_B, a_1\mu_B) + d_{\mg}(\mu_B, \omega \mu_B) +
      d_{\mg}(\mu_B,
      a_2\mu_B)\\
      & \prec d_{\mg}(\mu_B, f_1\mu_B) + d_{\mg}(\mu_B, f_2\mu_B). 
    \end{align*} 
    This concludes the proof the corollary.\qedhere
  \end{proof}
  
%section{Bibliography}

  \bibliographystyle{amsalpha}
  \bibliography{references}

  \end{document}